\documentclass[11pt]{article}
\usepackage{amssymb}
\usepackage{amsmath,amsthm}
\usepackage{float}
\usepackage{graphicx}
\usepackage{subcaption}

\newtheorem{theorem}{Theorem}[section]
\newtheorem{lemma}[theorem]{Lemma}

\newtheorem{definition}[theorem]{Definition}
\newtheorem{remark}[theorem]{Remark} 
\newtheorem{example}[theorem]{Example}

\begin{document}

\title{\bf Classification of Dupin Cyclidic Cubes by Their Singularities}

\author{
Jean Michel Menjanahary$^a$,
Eriola Hoxhaj$^b$,
Rimvydas Krasauskas$^a$ \\ \\
\em $^a$Institute of Computer Science, Vilnius University, Lithuania\\
\em $^b$Johannes Kepler University, Linz, Austria
}

\date{\empty}
\maketitle

\begin{abstract}
Triple orthogonal coordinate systems having coordinate lines as circles or straight lines are considered. 
Technically, they are represented by trilinear rational quaternionic maps and are called Dupin cyclidic cubes, naturally generalizing the bilinear rational quaternionic parametrizations of principal patches of Dupin cyclides.
Dupin cyclidic cubes and their singularities are studied and classified up to 
M\"obius equivalency in Euclidean space.
\end{abstract}

{\em Keywords:} {\small Dupin cyclide, Dupin cyclidic cube,  bicircular quartics}

\def\R{\mathbb R}       
\def\C{\mathbb C}       
\def\D{{\mathbb D}}
\def\H{\mathbb H}       
\def\Sp{\mathbb{S}}              
\def\P{\mathbb{P}}       
\def\Ss{S}              
\def\S{\mathcal{S}} 

\def\re{\mathop{\mathrm{Re}}}
\def\im{\mathop{\mathrm{Im}}}
\def\rank{\mathop{\mathrm{rank}}}
\def\inv{\mathop{\mathrm{Inv}}}
\def\Qdr{{\cal Q}}
\newcommand{\sing}{\mathop{\mathrm{Sing}}}


\def\ic{\mathrm{i}} 
\def\ii{\mathbf{i}}
\def\jj{\mathbf{j}}
\def\kk{\mathbf{k}}

\def\B{B\'ezier }
\def\M{M\"obius }

\section{Introduction}

Consider a triply orthogonal coordinate system in $\R^3$ where all coordinate lines are either circles or straight lines. The general coordinate surfaces of such a system will be covered by two families of circles (or straight lines) that are curvature lines. This is the characteristic property of Dupin cyclides.  
Therefore, we call such coordinate systems \emph{Dupin cyclidic (DC) systems}. 

In modeling applications, the focus is mainly on principal patches of Dupin cyclides, i.e., quadrangular patches bounded by four curvature lines, which are circles meeting orthogonally at the corners. 
A natural generalization of the latter surface patch to a volume object is the definition of a Dupin cyclidic cube as a volume cut out of $\R^3$ by six principal DC patches meeting orthogonally.
For various potential applications of DC cubes, it is essential to avoid singularities. 
This was the primary motivation for the current research.    

Though DC systems are very particular cases of triply orthogonal coordinate systems studied 
intensively from the 19th century to all our knowledge, the full \M classification presented in this paper was not done.
We find 4 types, where two types (A) and (B) seem to be earlier unknown. Their \M classes are characterized by symmetry properties and their singular sets, which are certain arrangements of quartic curves in space.   

Recently, DC systems were considered in the context of Lie sphere geometry in \cite{Szewieczek2022}.
Our approach is mostly close to \cite{BobenkoHuhnen2012}, where a DC cube is defined as an
elementary hexahedron of 3D cyclidic nets. 

The paper is structured as follows. 
Section~\ref{sec:preliminaries} contains preliminaries about quaternions, Study quadric, and quaternionic--\B formulas parametrizing Dupin cyclides. 
DC systems are defined, and the main classification problem is formulated 
in Section~\ref{sec:DC-systems}. 
Section~\ref{sec:spherical} deals with the simplest cases when at least one family of surfaces 
is consisting of spheres or planes.   
Offsets of Dupin cyclides are considered in Section~\ref{sec:offsets}. 
Section~\ref{sec:3P} is devoted to DC systems with three symmetry planes.
In Section~\ref{sec:General}, the case with two symmetry planes and one imaginary symmetry sphere is analyzed, and the main classification results formulated in Theorems~\ref{th:main} and \ref{th:main2} are proved. 
%
The notion of the degree of DC systems is considered in Section~\ref{sec:degrees}.
Conclusions and further research directions are outlined in Section~\ref{sec:conclusions}.

\section{Preliminaries}
\label{sec:preliminaries}
\subsection{Quaternions and \M transformations in $\R^3$}

We will use the algebra of quaternions $\H$ in the standard basis $\{1,\ii,\jj,\kk\}$, where $\ii^2\!=\!\jj^2\!=\!\kk^2\!=\!\ii\jj\kk\!=\!-1$.
For a quaternion $q = r + x \ii + y \jj + z \kk$, 
define the real part $\re(q)=r$, the imaginary part $\im(q)= q-\re(q)$,
the conjugate $\Bar{q}=\re(q)-\im(q)$, and the norm $|q|=\sqrt{q\Bar{q}}$. If $q\neq0$, we have $q^{-1}=\Bar{q}/|q|^2$.
The Euclidean space $\R^3$ is naturally identified with the space of imaginary quaternions $\im\!\H=\{q\in \H \mid \re(q)=0\}$.

An inversion $\inv_q^r$ with respect to a sphere of center $q\in \im\!\H$ and radius $r$ can be written explicitly as $\inv_q^{r}(p)=q-r^2(p-q)^{-1}$ for all $p\in \im\!\H$.
The group generated by inversion transformations is called the group of \M transformations in $\R^3$. To be more precise, \M transformations are defined on the extended space 
$\widehat{\R}^3=\R^3\cup \{\infty\}$, which is identified with $\im\!\widehat{\H}=\im\!\H \cup \{\infty\}$. For example, the inversion $\inv_0^1$ maps the origin $0$ to $\infty$ and vice versa. 
 Alternatively, \M transformations can be defined by three kinds of generators: translations $p \mapsto p+a$, $a \in \im\!\H$; homotheties $p \mapsto \lambda p$, $\lambda \in \R^*$; and the unit inversion $\inv_0^1:\; p \mapsto  -p^{-1}=p/|p|^2$.

Since circles and straight lines are \M equivalent, 
they both will be called M-circles throughout this paper. Similarly, both spheres and planes will be called M-spheres.

\subsection{The Study quadric and quaternionic--\B formulas}

Define the quadratic form $\S$ in $\R^8$, which is identified with $\H^2$, by
\begin{equation}\label{Study}
\S(u,w) = \frac{u \bar w + w \bar u}{2}, \quad (u,w)\in \H^2.    
\end{equation}
Let $\R P^7$ be the real projectivization of $\H^2\cong \R^8$. The quadric in $\R P^7$ defined by $\S(u,w) = 0$ is called the \emph{Study quadric}. 
Actually, the Study quadric is the preimage of $\im\!\widehat{\H}$ under the \emph{projective division}
\[
\pi: \R P^7 \to\im\!\widehat{\H}, \ 
\pi(u,w) = \begin{cases}
            u w^{-1} & w \ne 0,\\
            \infty   & w = 0.
           \end{cases}
\]

Quaternionic \B (QB) formulas are defined by fractions of quaternionic polynomials
(expressed in \B form) $F = U W^{-1}$, such that the image is completely contained in $\im\!\widehat{\H}$.
First, one can define the 
B\'ezier formulas
on the Study quadric with \emph{homogeneous control points}
\[
\begin{pmatrix} U \\ W \end{pmatrix}
= \sum_i 
\begin{pmatrix} u_i \\ w_i \end{pmatrix} B_i,\ (u_i, w_i) \in \H^2\setminus \{(0,0)\}, 
\]
where $\{B_i\}$ is a certain polynomial Bernstein basis, and then apply the projective 
division $\pi(U, W)=UW^{-1}$.
If the quaternionic coefficients $w_i$, which are called \emph{weights}, are non-zero
then the usual formula
with control points 
$p_i = u_i w_i^{-1} \in \im\!\H$ is obtained
\begin{equation}
\label{proj-div}
U W^{-1} = 
\left(\textstyle{\sum}_i p_i w_i B_i\right)
\left(\textstyle{\sum}_i w_i B_i\right)^{-1} \subset \im\!\widehat{\H}.   
\end{equation}
 This representation can be found in \cite[Section~2.3]{KrasZube2014}.

 \begin{remark}
 Note that QB formulas are preserved by \M transformations. In particular, the inversion $\inv_q^r$ maps a QB formula with homogeneous control points $(u_i,w_i)$ to a QB formula with homogeneous control points $(u_i',w_i')$ such that
 \begin{equation}
u_i' =qu_i-(r^2+q^2)w_i, \quad w_i'=u_i-qw_i.
 \end{equation}
 \end{remark}

\subsection{Dupin cyclides: parametrizations and implicit equations}
\label{subsec:DuC}

Dupin cyclides are surfaces characterized by the property that their curvature lines 
are M-circles. We call them principal circles to distinguish them from
other M-circles on the same surface. Rational parameterizations of quadrangular Dupin cyclide patches bounded by four principal circles are most important for applications.
These are principal Dupin cyclide patches, which can be parametrized by bilinear QB formulas, where its corner points are exactly control points. 
The details are presented in \cite{ZubeKras2015}.
In particular, all four control points are on an M-circle, see, e.g., \cite{BobenkoSuris2008}. 

\begin{remark}
\label{rem:cross-ratio}
The condition of four points $p_0,p_1,p_2,p_3\in \im\!\H$ being on an M-circle is equivalent to the cross-ratio
 \[cr(p_0,p_1,p_2,p_3)=(p_0-p_1)(p_1-p_2)^{-1}(p_2-p_3)(p_3-p_0)^{-1},\]
being real; see Lemma $2.3$ in $\cite{ZubeKras2015}$.  
\end{remark}


\begin{definition}
\label{def:Farin-point}
A \emph{Farin point} of a quaternionic circular arc 
\begin{equation}\label{eq:circle}
C: [0,1] \to \R^3, \quad C(t) = (p_0 w_0 (1-t) + p_1 w_1 t) (w_0 (1-t) + w_1 t)^{-1}
\end{equation}
is point $f = C(1/2)$ in the interior of the arc $C([0,1])$
with endpoints $p_0$ and $p_1$ that controls its rational parametrization.
The Farin point $f$ can be moved to any other interior point of the arc by changing $w_1$ to $\lambda w_1$, $\lambda \in \R$, $\lambda > 0$.
\end{definition}

\begin{definition}
\label{def:DC-patch}
A \emph{DC patch} is a rational map defined by the bilinear QB formula
\begin{equation}\label{DC-patch}
P: (\R P^1)^2 \to \im\!\widehat{\H}\cong \widehat{\R}^3, \ P = U W^{-1},\  U, W \in \H[s,t],
\end{equation}
with orthogonality condition $\partial_s P \perp \partial_t P$. 
\end{definition}

\begin{remark}
\label{rem:Farin-patch}
For a DC patch $P(s,t)$ with control points $p_0,p_1,p_2,p_3\in \im\!\H$ 
Farin points on opposite arcs, e.g., $f_{01} = P(1/2,0)$ and 
$f_{23} = P(1/2,1)$ are related, since they define a sub-patch with control
points  $p_0$, $f_{01}$, $p_2$, $f_{23}$. In particular, they are concyclic.
\end{remark}

\begin{theorem}\label{th:implicit}
The implicit equation of a surface parametrized by bilinear DC patch with homogeneous control points $(u_i,w_i)$ is a factor of the $4 \times 4$ determinant 
\begin{equation}\label{implicit}
f(x,y,z) = \det([X w_i - u_i], i = 0,\ldots,3), \ X = x \ii + y \jj + z \kk,
\end{equation}
where $[q]$ denotes the coordinate column of the quaternion $q$.
The unique exception $f(x,y,z) \equiv 0$ happens only when the DC patch is on an M-sphere and all coordinate M-circles intersect at one point.
\end{theorem}
\begin{proof}
This is a particular case of the result \cite[Theorem~4.5]{KrasZube2020} since
the patch coming from a B-plane in the Study quadric is \M equivalent to a bilinear QB patch
with real weights. In our case, this can only be a planar rectangular patch since its 
intersecting lines should be orthogonal.
\end{proof}

\begin{theorem}\label{th:Du-control-points}
The Dupin cyclide principal patch with the given corner points 
$p_0,p_1,p_2,p_3$ on an M-circle and orthogonal tangent vectors $v_1$ and $v_2$ 
at $p_0$ can be parametrized using the DC patch
with the following weights (or homogeneous control points):
\begin{itemize}
\item[(i)]
    $p_0 = \infty$ and $p_1,p_2,p_3$ are collinear, $p_1 \ne p_2$, then 
    \begin{equation}\label{inf-weights}
    \begin{pmatrix}u_i \\ w_i \end{pmatrix}_{i=0,\ldots,3} =
    \begin{pmatrix}
    1 & -p_1 v_1 & -p_2 v_2 & p_3 (p_1-p_2) v_1 v_2 \\ 
    0 & -v_1 & -v_2 & (p_1-p_2) v_1 v_2 
    \end{pmatrix}.
    \end{equation}
\item[(ii)]
all control points are finite, only $p_1$ and $p_2$ may coincide with $p_3$, then
\begin{align}\label{fin-weights}
w_0 &= 1, \ w_1=(p_1-p_0)^{-1} v_1,\ w_2=(p_2-p_0)^{-1} v_2,  \\    
w_3 &= (p_3-p_0)^{-1}\left((p_1-p_0)^{-1}-(p_2-p_0)^{-1}\right) v_1 v_2. 
\end{align}
\end{itemize}
\end{theorem} 
\begin{proof}
(i) Because of \M invariancy one can assume: $v_1 = \jj$, $v_2 = \kk$, and
the straight lines, one through $p_1$ with direction $v_1$, and the other through $p_2$ with direction $v_2$, 
are crossing the $x$-axis.
By subdividing the initial patch along particular parameter values in both directions
one can choose all three control points on the $x$-axis: $p_i = h_i \ii$, $i=1,2,3$.
Then the formula \eqref{inf-weights} gives the homogeneous control points
\begin{equation*}
    \begin{pmatrix}u_i \\ w_i \end{pmatrix}_{i=0,\ldots,3} =
    \begin{pmatrix}
    1 & -h_1 \kk & h_2 \jj & -(h_1 - h_2) h_3 \ii \\ 
    0 &  -\jj & -\kk & -h_1 + h_2
    \end{pmatrix}.
\end{equation*}
Using the determinant \eqref{implicit}, the implicit equation can be computed 
\[
F_{U,W} = (h_1 - h_2)\left((x-h_1)(x-h_2)(x-h_3) + (x-h_1) y^2 + (x-h_2) z^2\right).
\]
This is indeed a Dupin cyclide: by substituting $h_1 = -r-p/2$, $h_2 = -r + p/2$,
$h_3 = r$ the equation of parabolic cyclide $F_P(x,y,z,r)$ (see \cite{chandru1989geometry}) 
is obtained. 

(ii) This general case is reduced to the previous item (i) by applying the inversion $\inv_{p_0}^1$ with center in the point $p_0$. 
\end{proof}

Bilinear DC patches can parametrize M-spheres or can degenerate to M-circles or 
isolated points when their Jacobian vanishes everywhere.  

\begin{lemma}\label{lem:spherical-cond}
The non-degenerated bilinear DC patch is an M-sphere if and only if
the following two equivalent conditions on its opposite homogeneous control points are satisfied
\begin{equation}\label{spherical-cond}
\S(u_i,w_j) + \S(u_j,w_i) = 0, \ (i,j) = (0,3), (1,2).    
\end{equation}
\end{lemma}
\begin{proof}
Directly follows from \cite[Lemmas~3.1 and 3.2]{KrasZube2020}.    
\end{proof}

\begin{remark}
\label{rem:spherical-cond}
The condition \eqref{spherical-cond} is satisfied when
the $2$-linear QB surface degenerates to a point, but it might be non-zero in some
cases of degeneration to an M-circle.
\end{remark}

\subsection{Bicircular quartics}
A real plane algebraic curve of degree 4, that doubly covers the circular points $(0\!:\!1\!:\!\pm \ic)$ at infinity, where $\ic$ is the imaginary unit, is called a bicircular quartic. Their implicit equation has the form
\begin{equation}
\label{eq:general_bq}
\lambda (x^2+y^2)^2+L(x,y)(x^2+y^2)+Q(x,y)=0,
\end{equation} 
where $\lambda$ is a constant, $L$ is a linear form and $Q$ is a quadratic polynomial in $x$ and $y$. 
The curve \eqref{eq:general_bq} is a circular cubic if $\lambda=0$ and a conic if $\lambda=L=0$.

If a bicircuar quartic $B$ is preserved by an inversion w.r.t. an M-circle $C$, then $C$ is called an M-circle of symmetry of $B$. The following well-known result will be useful in this paper.
\begin{theorem}
\label{th:bq-properties}
A bicircular quartic has $4$ mutually orthogonal M-circles of symmetry, and at least two of these M-circles are real. If the curve has one oval, then the other two circles are complex conjugated circles. If the curve has two ovals, three of the circles are real and the other one is imaginary.
\end{theorem}
\begin{proof}
This is proved in \cite[p.~304]{Hilton1932}.
\end{proof}

The two real M-circles of symmetry of a bicircular quartic can be brought by \M transformations to the $x$-axis and $y$-axis, and the equation of the curve reduced to the canonical symmetric form
\begin{equation}
\label{eq:bq-canonic}
{\cal B}^{\delta}: (x^2+y^2)^2-2Kx^2+2My^2+\delta=0, \qquad \delta=\pm 1, \quad K,M\in \R.
\end{equation}
For $\delta=-1$ (resp. $\delta=+1$), this curve ${\cal B}^{\delta}$ has one oval (resp. two ovals).

\begin{definition}
\label{def:foc}
Three symmetric bicircular quartics on the orthogonal coordinate planes 
\begin{align}
\label{eq:B_1}
{\cal B}_1^{\delta}: z=0,\ (x^2+y^2)^2 - 2K x^2 + 2M y^2 + \delta=0,\\
\label{eq:B_2}
{\cal B}_2^{\delta}: y=0,\ (x^2+z^2)^2 - 2M z^2 + 2N x^2 + \delta=0,\\
\label{eq:B_3}
{\cal B}_3^{\delta}: x=0,\ (y^2+z^2)^2 - 2N y^2 + 2K z^2 + \delta=0,
\end{align}
are called focal bicircular quartics if the coefficients of
their equations satisfy
\begin{equation}
\label{eq:foc-rel}
KM + MN + NK + \delta = 0.
\end{equation}
Moreover, for each $i=1,2,3$, the intersection points between the plane of ${\cal B}_i^{\delta}$ with the other two focal curves are called focal points of ${\cal B}_i^{\delta}$.
\end{definition}
In order to have simple expressions for focal points, we change parameters
\begin{equation}
\label{eq:foc-param}
K=-\frac{c\delta+c^{-1}}{2}, \quad N=-\frac{a\delta+a^{-1}}{2}, \quad M=-\frac{b\delta+b^{-1}}{2}, \quad \delta=\pm 1,
\end{equation}
where $b=-(a+c)(1+ac\delta)$ is computed using \eqref{eq:foc-rel}.
The \M canonical forms of 1-oval bicircular quartics ($\delta=-1$) are defined by $a,c>0$ and $ac>1$ with a clear geometric meaning: the curve ${\cal B}_1^{-}$ has $4$ real focal points, two of them on the $x$-axis and the other two on $y$-axis. Similarly, for ${\cal B}_2^{-}$ and ${\cal B}_3^{-}$ of different coordinate axes. The set of focal points of ${\cal B}_1^{-}$ and ${\cal B}_2^{-}$ are
\begin{equation}
\label{eq:foc-minus}
\Phi_1^- = \{ \pm \ii/\sqrt{a}, \pm\jj\sqrt{a}\},\quad \Phi_2^- = \{ \pm \ii\sqrt{c}, \pm\kk/\sqrt{c}\}.
\end{equation}
The focal points of ${\cal B}_3^{-}$ can also be computed symmetrically. The condition $ac>1$ keeps the focal points of ${\cal B}_1^{-}$ inside the oval. 
Note that the other two focal curves ${\cal B}_2^{-}$ and ${\cal B}_3^{-}$ are all 1-oval curves in this case; see the left side of Figure~\ref{fig:1-2-foc-bqs}.
On the other hand, the \M canonical forms of 2-oval bicircular quartics ($\delta=+1$) are defined by $c<0$ and $a>|c|>1/a$ with the following geometric meaning: the $4$ real focal points of ${\cal B}_1^{+}$ and ${\cal B}_2^{+}$ are lying on the same $x$-axis:
\begin{equation}
\label{eq:foc-plus}
\Phi_1^+ = \{ \pm \ii/\sqrt{a}, \pm\ii\sqrt{a}\},\quad \Phi_2^+ = \{ \pm \ii\sqrt{-c}, \pm\ii/\sqrt{-c}\}.
\end{equation}
The curve ${\cal B}_3^{+}$ has no real points in this case. The condition $a>|c|>1/a$ restricts ${\cal B}_1^{+}$ to a curve with enclosed ovals and ${\cal B}_2^{+}$ to a curve with two separated ovals, see the right side of Figure~\ref{fig:1-2-foc-bqs}.

\begin{figure}[H]
    \centering
    \includegraphics[width=0.45\textwidth]{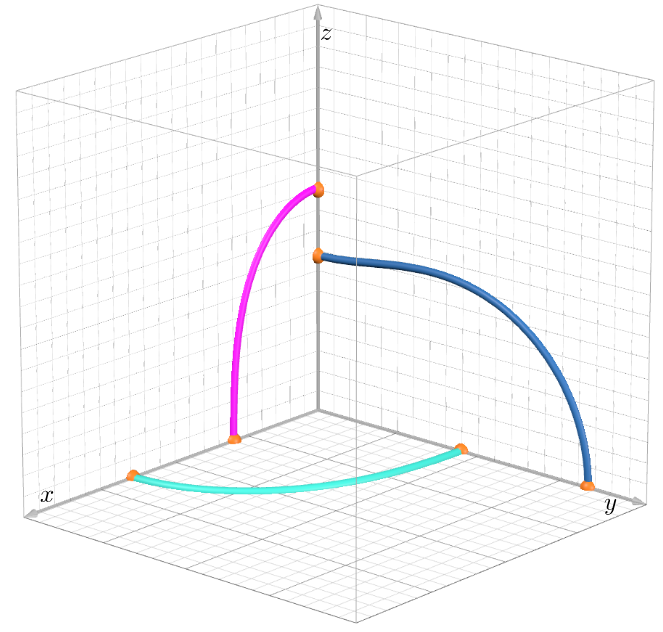} \qquad
    \includegraphics[width=0.45\textwidth]{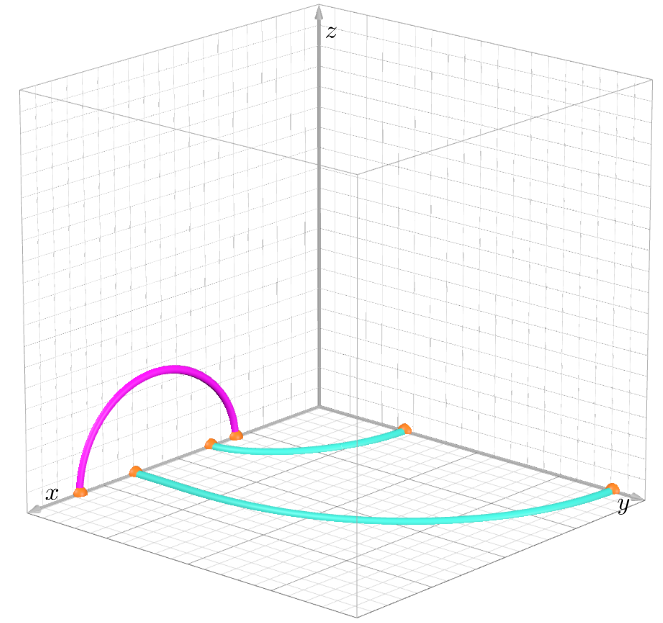}
    \caption{Focal symmetric bicircular quartics are depicted in quarters on the first octant of the Euclidean space: 1-oval curves on the left and 2-oval curves on the right.}
    \label{fig:1-2-foc-bqs}
\end{figure}




\section{Dupin cyclidic systems}
\label{sec:DC-systems}

\begin{definition}
\label{def:DC}
A \emph{Dupin cyclidic (DC) system} in $\R^3$ as the $3$-linear rational quaternionic map
to the imaginary quaternions
\[
F: (\R P^1)^3 \to \im\!\widehat{\H}\cong \widehat{\R}^3, \quad F = U W^{-1}, \ U, W \in \H[s,t,u],
\]
such that: all three partial derivatives $\partial_s F$, 
$\partial_t F$, $\partial_u F$ are mutually orthogonal and the Jacobian $\mathrm{Jac}(F)$ is non-zero at least in one point.
\end{definition}

Here $\widehat{\R}^3 = \R^3 \cup \{\infty\}$ is treated as 3-dimensional sphere $S^3$ and $F$ is a smooth map between differential manifolds. 
Therefore, any differential properties of $F$ at the infinite point $\infty$ should be computed for the map $\inv_0^1 \circ F$ at the origin.

\begin{definition}
\label{def:M-equiv}
Two DC systems $F$ and $F'$ are \emph{\M equivalent} if and only if
$F' = \mu \circ F \circ \rho$, where $\mu$ is a \M transformation of $\widehat{\R}^3$ and $\rho$ is an algebraic automorphism of $(\R P^1)^3$
generated by projective transformations of lines $\R P^1$ and their permutation.
\end{definition}

Any DC system $F$ defines three families of surfaces in $\R^3$: namely
$s$-surfaces $F_{s**} = \{ F(s,t,u) \mid t, u \in \R P^1\}$,
$t$-surfaces $F_{*t*}$, and $u$-surfaces $F_{**u}$, defined in similar way. 

\begin{definition}
\label{def:sing}
The \emph{singular locus} $\sing(F) \subset \widehat{\R}^3$ of DC system $F$ 
is the image of all points where its Jacobian vanishes.
Define  $\sing_i(F) \subset \R^3$, $i=1,2,3$,
as images of sets where $\partial_s F=0$, $\partial_t F=0$, $\partial_u F=0$,
respectively. 
\end{definition}

\begin{lemma}
\label{lem:DC-systems}
Singular sets of a DC system $F$ have the following properties:
\begin{itemize}
\item[(i)] if $p \in \sing_i(F)$, $i=1,2,3$, then $F^{-1}(p)$ 
contains a line in the corresponding direction of $(\R P^1)^3$;
\item[(ii)]  $\sing(F) = \sing_1(F) \cup \sing_2(F) \cup \sing_3(F)$;
\item[(iii)] $\sing_1(F) \subset F_{s**}$,  $\sing_2(F) \subset F_{*t*}$, 
$\sing_3(F) \subset F_{**u}$.  
\end{itemize}
\end{lemma}
\begin{proof}
(i) follows from the linearity of the quaternionic formula when restricted to a line
in $(\R P^1)^3$. (ii) follows from the orthogonality of partial derivatives.
Finally, item (iii) follows from (i).  
\end{proof}

Representing the map $F$ in \B form will be useful. 
Define the associated \emph{Dupin cyclidic (DC) cube} as a map 
$D: [0,1]^3 \to \H^2$ such that $\pi \circ D = F$ is given in the \B form
\begin{equation}
\label{eq:DC-cube-QB}
D(s,t,u) = 
\begin{pmatrix}
U(s,t,u) \\ W(s,t,u)
\end{pmatrix}
= \sum_{i=0}^1\sum_{j=0}^1\sum_{k=0}^1 
\begin{pmatrix}
u_{ijk} \\ w_{ijk}
\end{pmatrix}
B_i^1(s)B_j^1(t)B_k^1(u),
\end{equation}
where $B_0^1(t) = 1-t$,  $B_1^1(t) = t$ are linear Bernstein polynomials.  
Here homogeneous control points $(u_{ijk},w_{ijk}) = (U(i,j,k),W(i,j,k))$ define
control points $p_{ijk} = u_{ijk} w^{-1}_{ijk}$ in $\im\,\H$ if $u_{ijk} \ne 0$
($p_{ijk} = \infty$ otherwise), $i,j,k \in \{0,1\}$.
Alternative indexing of control points also will be used $p_0 = p_{000},p_1 = p_{100},p_2 = p_{010},\ldots,p_7 = p_{111}$.
This construction directly generalizes the bilinear QB patch described in 
Section~\ref{subsec:DuC} and \cite{ZubeKras2015}.

\begin{remark}
\label{rem:DCSphereCond}
A DC cube $D$ can be sliced into families of DC patches $D_{s**}$, $D_{*t*}$, $D_{**u}$; e.g., for a particular value of $s$ the patch $P = D_{s**}$ will have control points $(u_{sij},w_{sij}) = (U(s,i,j),W(s,i,j))$, $i,j=0,1$, and similarly in $t$- and $u$-directions. Then, Lemma~\ref{lem:spherical-cond} can be used to detect when a slice patch degenerates to an M-sphere in three directions:
$\sigma_1(s)=0$, $\sigma_2(t)=0$, $\sigma_3(u)=0$, where
\begin{align*}
\sigma_1(s)&=\S\big(u_{s00},w_{s11}\big)+\S\big(u_{s11},w_{s00}\big), \ 
\sigma_2(t)=\S\big(u_{0t0},w_{1t1}\big)+\S\big(u_{1t1},w_{0t0}\big),\\
\sigma_3(u)&=\S\big(u_{00u},w_{11u}\big)+\S\big(u_{11u},w_{00u}\big).
\end{align*}
\end{remark}
Remarkably, all the 8 control points of the cube are on an M-sphere.
The existence of the point $p_7$ on the M-sphere can be derived using Miquel's theorem 
on a triangle \cite{miquel1844memoire}.

\begin{theorem}[Miquel's Theorem]
\label{th:miquel}
The $3$ circles, each defined by a vertex of the triangle and two points on the adjacent sides (Figure $\ref{fig:miquel}$), intersect in one point, called the {\em Miquel point}. 
\end{theorem}

\begin{figure}[H]
    \centering
    \includegraphics[width=4.7cm]{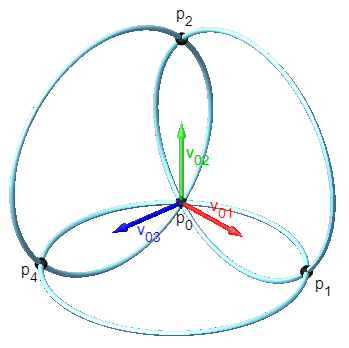}\qquad
    \includegraphics[width=5.6cm]{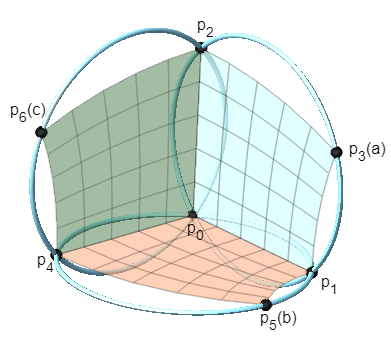} 
    \includegraphics[width=4.9cm]{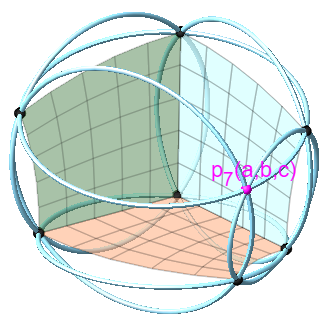}\qquad
     \includegraphics[width=5.2cm]{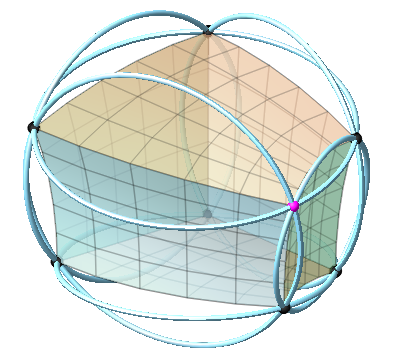}
    \caption{Four steps to build a general Dupin cyclidic cube.}
    \label{fig:unit_cube}
\end{figure}

To construct a DC cube based on 3 given faces, we apply inversion $\inv_{p_0}^1$ at $p_0$ so that all $7$ control points are coplanar and the other one on infinity. We compute the Miquel point $M$ on the triangle $q_1q_2q_4$ with side points $q_3$, $q_5$, $q_6$, where $q_i=\inv_{p_0}^1(p_i)$, and apply the same inversion to obtain $p_7=\inv_{p_0}^1(M)$.

\begin{lemma}[Miquel point]
\label{lem:miquel}
 Let $p_1$, $p_2$, $p_3\in \im\!\H$ be three generic points. Let $d_i=\lVert p_j-p_k\rVert^2$ and let $q_i$ be a point on a side of the triangle $p_1p_2p_3$ such that $q_i=\lambda_i'p_j+\lambda_ip_k$, where $i,j,k\in \{1,2,3\}$ pairwise distinct and $\lambda_i\in \R$, $\lambda_i'=1-\lambda_i$. 
 Then the Miquel point $M$, expressed in baricentric coordinates, is given by
 \[M=\frac{\sum_{i=1}^3 p_i\alpha_i}{\sum_{i=1}^3\alpha_i},\]
 where 
\begin{equation*}
\begin{cases}
\alpha_1 = -\lambda_1\lambda_1'd_1^2+\lambda_1\lambda_2d_1d_2+\lambda_1'\lambda_3'd_1d_3,\\
\alpha_2 = \lambda_1'\lambda_2'd_1d_2-\lambda_2\lambda_2'd_2^2+\lambda_2\lambda_3d_2d_3,\\
\alpha_3 = \lambda_1\lambda_3d_1d_3+\lambda_2'\lambda_3'd_2d_3-\lambda_3\lambda_3'd_3^2.
\end{cases}
\end{equation*}
\end{lemma}
\begin{proof}
Up to Euclidean similarities, we assume $p_1=0$, $p_2=\ii$, $p_3=k\ii+m\jj$, where $k,m\in \R$. Then $d_1=m^2+(k-1)^2$, $d_2=k^2+m^2$ and $d_3=1$.
We have the real symbolic cross-ratios
\begin{align*}
cr(p_1,q_2,q_3,M) &= \frac{\lambda_2'H_M}{\lambda_2'^2m^2+(\lambda_2'k-\lambda_3)^2},\\
cr(p_2,q_1,q_3,M) &= \frac{\lambda_1H_M}{\lambda_1^2m^2+(\lambda_1k+\lambda_1'-\lambda_3)^2},
\end{align*}
where $H_M=(\lambda_1\lambda_3+\lambda_2'\lambda_3')d_2-(2\lambda_1k+\lambda_1'-\lambda_3)\lambda_3$.
From Remark \ref{rem:cross-ratio} and Theorem \ref{th:miquel}, it follows that the formula for $M$ defines the Miquel point.
\end{proof}

\begin{figure}[H]
    \centering
    \includegraphics[width=6cm]{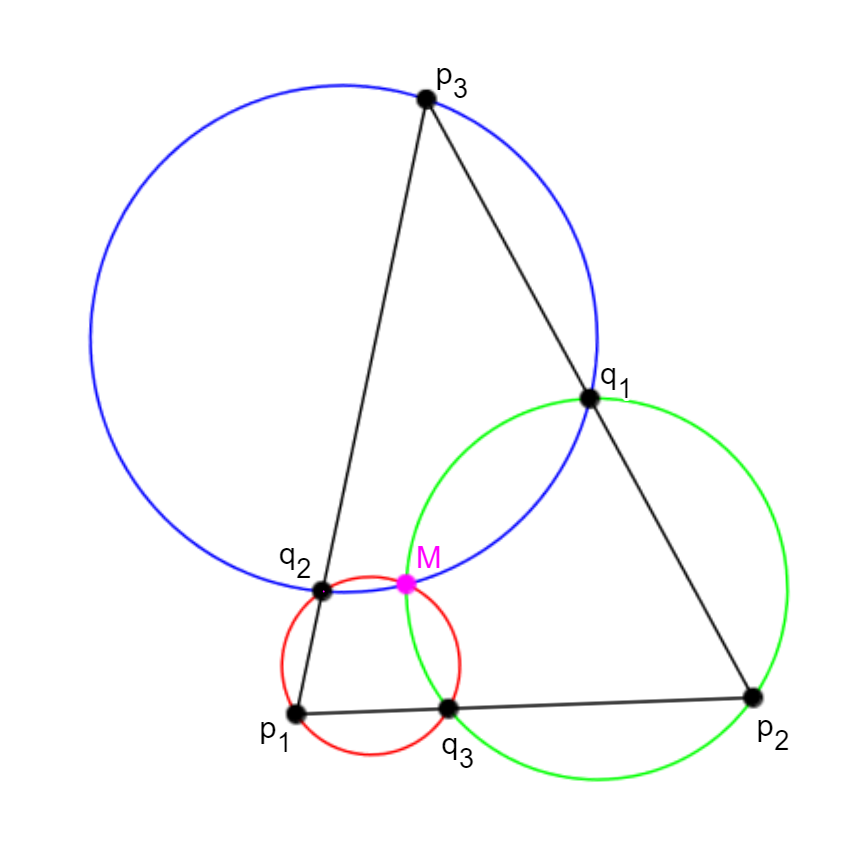}
    \caption{Miquel point.}
    \label{fig:miquel}
\end{figure}

\begin{lemma}
\label{lem:MiquelFarin}
A DC cube can be uniquely built from the compatible DC patches on its three adjacent faces, i.e., when their parametrizations on common arcs coincide. 
\end{lemma}
\begin{proof}
Farin points (see Definition~\ref{def:Farin-point}) on opposite boundary arcs of 
a DC patch $P(s,t)$ are related according to Remark~\ref{rem:Farin-patch}. 
Their positions visualize reparametrizations in two directions that do not 
change the surface and boundaries of the patch.
Suppose we have bilinear QB parametrizations of three faces of a DC cube $D$ incident with $p_0$ that are compatible on common edges. 
This means that their Farin points $D(1/2,0,0)$, $D(0,1/2,0)$, and $D(0,0,1/2)$ 
are compatible on the corresponding edges.
They uniquely determine 6 other Farin points on opposite edges of these initial 
three DC patches, e.g., the point $D(1/2,0,0)$ determines two others:
$D(1/2,1,0)$ and $D(1/2,0,1)$. Then we compute the Miquel point $p_7$ and 
add three new DC patches, which have already been prescribed Farin points on 
the couples of old edges so that their parametrizations are uniquely defined. 
Are they compatible along the last three edges incident with $p_7$? 
The answer is positive and follows directly from Miquel Theorem. 
\end{proof}

\begin{remark}\label{rem:interior-rep}
Moving Farin points on three edges incident with $p_0$ will define the reparametrization of the DC cube, which is equivalent to the multiplication of 
all $8$ homogeneous control points by real nonzero multipliers in the following order: 
$1$, $\lambda_1$, $\lambda_2$, $\lambda_1\lambda_2$, $\lambda_3$, 
$\lambda_1\lambda_3$, $\lambda_2\lambda_3$, $\lambda_1\lambda_2\lambda_3$.
This process will be called interior reparametrization with factor $(\lambda_1,\lambda_2,\lambda_3)$  of the DC cube or system.
\end{remark}

\section{Spherical Dupin cyclidic systems}
\label{sec:spherical}
A DC system is called \emph{spherical} if at least one of its families of surfaces consists 
of M-spheres. 
It is useful to understand how the coordinate lines on the M-spheres behave. 

\subsection{Two dimensional DC systems}
\label{subsec:2d-DC}

Using M-circles as coordinate lines, there are two classical orthogonal coordinates on the plane: Cartesian and polar systems. The Cartesian coordinates have a pole (or singularity) on infinity, and the polar coordinate has two poles (one on infinity). More examples can be obtained by applying \M transformations. We call those two types of coordinates {\em $1$-polar} and {\em $2$-polar} coordinates on an M-sphere, depending on the number of poles.

\begin{lemma}
\label{lem:2d-dc}
A $2$-dimensional DC system is \M equivalent to a $1$-polar or a $2$-polar coordinate system.
\end{lemma}
\begin{proof}
Consider one family of coordinate lines that are M-circles on an M-sphere. Depending on their intersections, there are 3 possibilities to position 2 M-circles from the family. 
Under \M transformations, we map the case of 2 non-intersecting M-circles to 2 concentric circles, 2 intersecting M-circles to 2 intersecting lines, and 2 touching M-circles to 2 parallel lines on a plane. 
From the orthogonality condition, the other coordinate lines are uniquely constructed to fill $\R^2$; see Figure~\ref{fig:2Dcase}.
\end{proof}

\begin{figure}[H]
    \centering
    \begin{subfigure}[b]{0.3\textwidth}
        \centering
        \includegraphics[width=\textwidth]{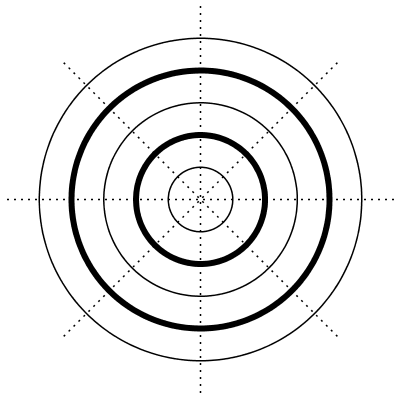}
        \caption{2 concentric circles (2-polar)}
    \end{subfigure}
    \hfill
    \begin{subfigure}[b]{0.3\textwidth}
        \centering
        \includegraphics[width=\textwidth]{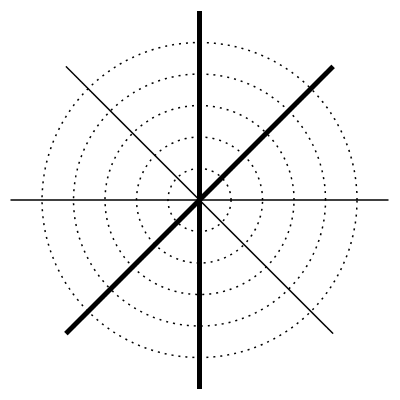}
        \caption{2 intersecting lines (2-polar)}
    \end{subfigure}
    \hfill
    \begin{subfigure}[b]{0.3\textwidth}
        \centering
        \includegraphics[width=\textwidth]{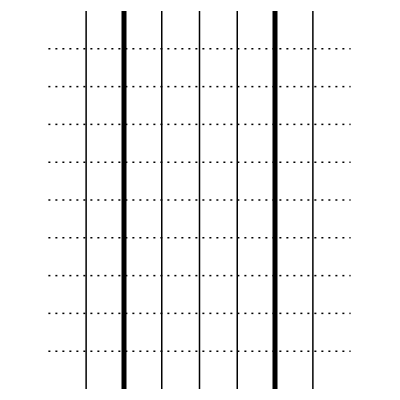}
        \caption{2 parallel lines (1-polar)}
    \end{subfigure}
    \caption{The 2 initial M-circles are in bold, followed by solid M-circles to fulfill the family. The dotted M-circles form the orthogonal family.}
    \label{fig:2Dcase}
\end{figure}

\begin{example}\rm 
\label{exa:bipolar}
Assume that on the $xy$-plane, the $x$-axis and $y$-axis are already coordinate lines of a $2$-dimensional DC system. We choose the control points $p_0=0$, $p_1=\ii$, $p_2=\jj$, and the fourth control point $p_3$ is on the circle define by those points, namely
\[
p_3 = \frac{1+a}{1+a^2}\ii + \frac{1-a}{1+a^2}\jj, \quad a\in \R.
\]
The corresponding weights are computed using Theorem $\ref{th:Du-control-points}$, giving us the homogeneous control points
 \begin{equation}
    \begin{pmatrix}u_i \\ w_i \end{pmatrix}_{i=0,\ldots,3} =
    \begin{pmatrix}
    0 &  \ii & \jj & \ii+\jj \\ 
    1 &  1 & 1 & 1-a \kk
    \end{pmatrix}.
\end{equation}
The associated QB parametrization is $F(s,t)=(s\ii+t\jj)(1-ast\kk)^{-1}$. This construction is illustrated in Figure $\ref{fig:2-polar-3P}(a)$.
The map $F$ is a double covering of the $xy$-plane, where the two preimages are related by the involution $(s,t)\mapsto (1/as,-1/at)$ on $(\R P^1)^2$.
Note that $a=0$ corresponds to the Cartesian coordinate system, and its inverse is shown in Figure $\ref{fig:2-polar-3P}(b)$, and defined by the homogeneous control points
 \begin{equation}
 \label{eq:1-polar}
    \begin{pmatrix}u_i \\ w_i \end{pmatrix}_{i=0,\ldots,3} =
    \begin{pmatrix}
    0 &  \ii & 0 & \ii \\ 
    1 &  1 & 1 & 1 +\kk
    \end{pmatrix}.
\end{equation}
To compute the weights of such a representation, one can 
apply Theorem~\ref{th:Du-control-points}(ii) for the permuted 
indices $(0,1,2,3) \mapsto (1,0,3,2)$ of the control points. 
\end{example}

\begin{figure}[H]
    \centering
    \begin{subfigure}[b]{0.4\textwidth}
        \centering
        \includegraphics[width=\textwidth,height=0.9\textwidth]{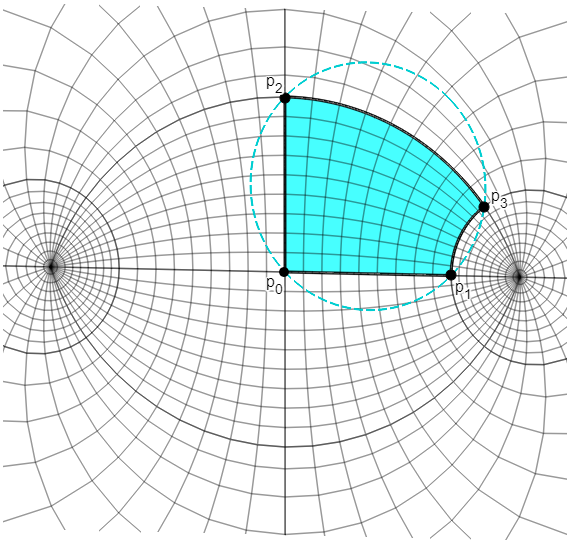}
        \caption{2-polar coordinate system on the $xy$-plane with $a=-1/2$ in Example \ref{exa:bipolar}.}
    \end{subfigure}
    \hskip1cm
    \begin{subfigure}[b]{0.4\textwidth}
        \centering
        \includegraphics[width=0.9\textwidth,height=0.9\textwidth]{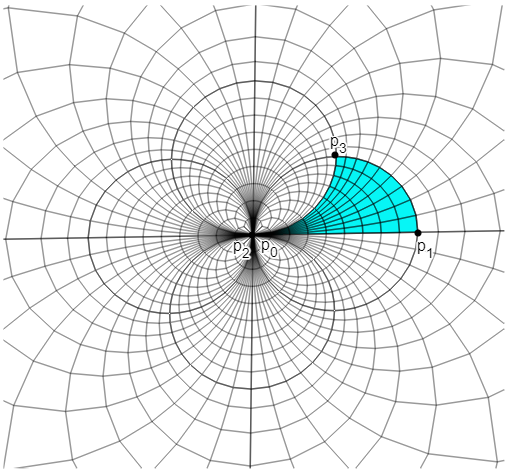}
        \caption{1-polar coordinate system, inverse of the Cartesian coordinate system at $a=0$.}
    \end{subfigure}
    \caption{Construction of a 2-dimensional DC system using control points.}
    \label{fig:2-polar-3P}
\end{figure}

\subsection{Spherical DC systems}

There are two natural ways to construct a spherical DC system from a 2-dimensional DC system, namely the axial and offset based constructions on the M-sphere. 
The axial case considers the rotation of a 2-dimensional DC system on a plane about a (straight) line on that plane. Assume that a DC patch on the plane is defined by homogeneous control points $(u_i,w_i)$, $i=0,1,2,3$. Then we can build an axial DC cube by left-multiplying the homogeneous control points with the direction of the line. More precisely, if $n$ is the direction of the line, then the additional control points are given by
$(u_{i+4},w_{i+4})=(nu_i,nw_i)$, $i=0,1,2,3$.

The offset based construction applies to any Dupin cyclides as we will describe in Section \ref{sec:offsets}. A DC cube is naturally constructed from a DC patch by taking offsets along normal directions at a fixed distance. The control points of the cube is obtained by offsetting the vertices of the DC patch and using the same weights as described in Lemma \ref{lem:offset-construction}.

\begin{theorem}
\label{th:spherical}
Any spherical DC system is \M equivalent to a DC system obtained from an axial or offset construction based on an M-sphere. They are classified as following:
\begin{itemize}
\item[$(S1)$] Offset construction based on a sphere. There is a one parameter family of \M classes where the singular locus is $2$ intersecting lines, and two limit cases where the singular locus is a double line exist.

\item[$(S2)$] Offset construction based on a plane. Four cases distinguished by the singular locus: $2$ parallel lines, a double line, a line or just a point.

\item[$(S3)$] Axial construction from $2$-polar system. There is one parameter family of \M classes where the singular locus consists of a double line and $2$ circles. Two limit cases where the singular locus is: a double line or a double line and a double circle.

\item[$(S4)$]  Axial construction from $1$-polar system. Two cases distinguished by the singular locus: a double line or a double line and a double circle.
\end{itemize}
\end{theorem}

\begin{proof}
Depending on the intersection of the 2 M-spheres from the same family, a \M transformation in $\R^3$ maps them to either concentric spheres, intersecting planes, or parallel planes. The unique choice of positioning M-circles orthogonal to the given 2 M-spheres (as in the earlier 2-dimensional construction) clarifies the reduction to axial or offset based construction. The classification is based on the choice of 1-polar or 2-polar coordinate system on the considered M-sphere. 

We parametrize generic 2-polar coordinates on the unit sphere using one parameter $a$. The offset based construction has homogeneous representation
\[
\begin{pmatrix}
\ii & \ii+\jj & \ii+\kk & a+\ii+\jj+\kk & -\ii & -\ii-\jj & -\ii-\kk & -a-\ii-\jj-\kk\\
1 & 1+\kk & 1-\jj & 1-a\ii-\jj+\kk & 1 & 1+\kk & 1-\jj & 1-a\ii-\jj+\kk
\end{pmatrix}.
\]
The system is indeed spherical, since the spherical condition $\sigma_3(u)$ (see Remark \ref{rem:DCSphereCond} is identically zero.
For an arbitrary value of $a$, we obtain two intersecting lines as singularities of the system, see Figure~\ref{fig:offset-spheres}(a).
The cases $a=1$ and $a=0$ correspond to the classical spherical system and the 1-polar system respectively. They cover the case (S1).

One can introduce similarly the other choices of 1-polar or 2-polar system and the offset or axial based construction, and obtained the properties of singularities of the system. Those systems are described by (S2), (S3) and (S4) and illustrated by Figures~\ref{fig:offset-planes}, \ref{fig:axial-non-sym} and \ref{fig:axial-1-polars}.
For example, in the case of offset of the 1-polar system defined by \eqref{eq:1-polar}, the homogeneous control points are given by
\begin{equation}
\label{eq:offset-1-polar}   
\begin{pmatrix}
 0 & \ii & 0 & \ii & \kk & \ii + \kk & \kk & -1 + \ii + \kk \vspace{0.1cm} \\ 
1 & 1 & 1 & 1+\kk & 1 & 1 & 1  & 1 + \kk
\end{pmatrix}.
\end{equation}
This is illustrated by Figure~\ref{fig:offset-planes}(c).

\end{proof}

\begin{figure}[H]
    \centering
    \begin{subfigure}[b]{0.3\textwidth}
        \centering
        \includegraphics[width=\textwidth]{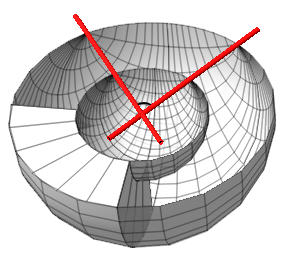}
        \caption{Offset of a generic 2-polar system on a sphere.}
    \end{subfigure}
    \hfill
    \begin{subfigure}[b]{0.32\textwidth}
        \centering
        \includegraphics[width=\textwidth]{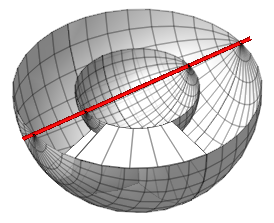}
        \caption{The 2 poles and the sphere's center are aligned.}
    \end{subfigure}
    \hfill
    \begin{subfigure}[b]{0.3\textwidth}
        \centering
        \includegraphics[width=\textwidth]{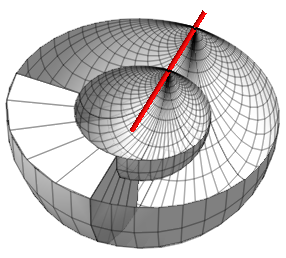}
        \caption{Offset of a 1-polar system on a sphere.}
    \end{subfigure}
    \caption{Spherical DC systems and their singularities (in red) based on sphere offsets.}
    \label{fig:offset-spheres}
\end{figure}

\begin{figure}[H]
    \centering
    \begin{subfigure}[b]{0.3\textwidth}
        \centering
        \includegraphics[width=1.2\textwidth]{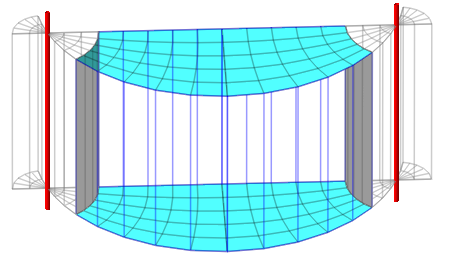}
        \caption{Offset of a generic 2-polar system.}
    \end{subfigure}
    \hfill
    \begin{subfigure}[b]{0.3\textwidth}
        \centering
        \includegraphics[width=0.8\textwidth,height=0.7\textwidth]{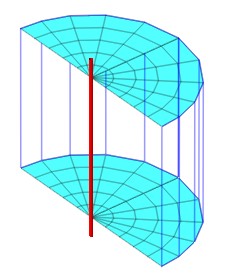}
        \caption{Offset of a 2-polar with one pole on infinity.}
    \end{subfigure}
    \hfill
    \begin{subfigure}[b]{0.3\textwidth}
        \centering
        \includegraphics[width=\textwidth]{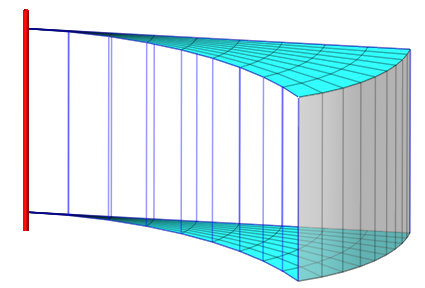}
        \caption{Offset of a 1-polar system.}
    \end{subfigure}
    \caption{Spherical DC systems and their singular curves based on offsets of a plane.}
    \label{fig:offset-planes}
\end{figure}

\begin{figure}[H]
    \centering
    \begin{subfigure}[b]{0.45\textwidth}
        \centering
        \includegraphics[width=\textwidth]{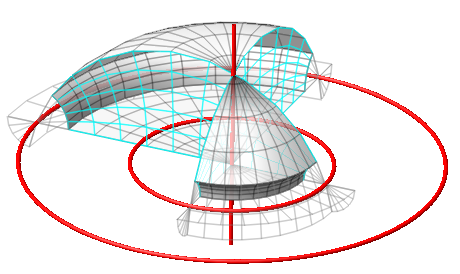}
        \caption{Rotating a 2-polar system on a plane about a generic axis on the plane. This represents all \M classes of all such systems.}
    \end{subfigure}
    \hfill
    \begin{subfigure}[b]{0.45\textwidth}
        \centering
        \includegraphics[width=\textwidth]{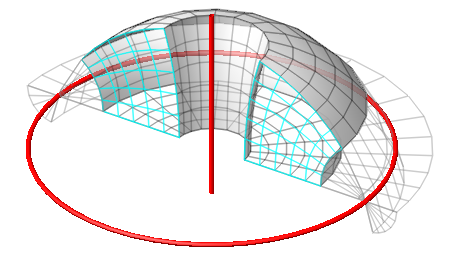}
        \caption{Rotating a 2-polar system on a plane about an axis on the plane, where the two poles are placed symmetrically with respect to the axis.}
    \end{subfigure}
    \caption{Spherical DC systems obtained by axial-based construction of a 2-polar system on a plane. Their singularities are concentric circles and a double line.}
    \label{fig:axial-non-sym}
\end{figure}

\begin{figure}[H]
    \centering
    \begin{subfigure}[b]{0.42\textwidth}
        \centering
        \includegraphics[width=\textwidth]{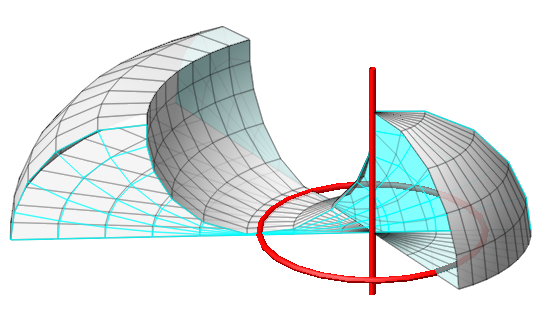}
        \caption{Rotating a 1-polar system about an axis not passing through the pole.}
    \end{subfigure}
    \hfill
    \begin{subfigure}[b]{0.42\textwidth}
        \centering
        \includegraphics[width=\textwidth]{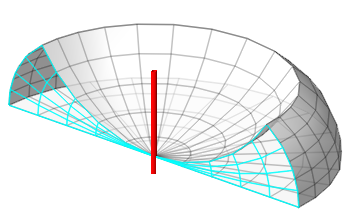}
        \caption{Rotating a 1-polar system about an axis passing through the pole.}
    \end{subfigure}
    \caption{Spherical DC systems obtained by axial based construction of a 1-polar system on a plane. Their singularities are a double circle and a double line. The circle degenerates to a point in the case (b).}
    \label{fig:axial-1-polars}
\end{figure}


\section{Non-spherical offset Dupin cyclidic systems}
\label{sec:offsets}

A DC system is called {\em offset DC system} if it is M\"obius equivalent to the DC system obtained by offsetting of a Dupin cyclide. 
To construct an offset cube from a Dupin cyclide, we use the following result.

\begin{lemma}
 \label{lem:offset-construction}  
 Let a Dupin cyclide be defined by $4$ control points $p_i$, $i=0,1,2,3$, 
and orthogonal tangent vectors $v_1$ and $v_2$ at $p_0$. 
Let $w_i$ be the corresponding weight computed using Theorem $\ref{th:Du-control-points}$.
The normal at $p_0$ is $n_0=v_1v_2$. 
The normal at $p_i$ is $n_i=w_in_0w_i^{-1}$, $i=1,2,3$.
The offset at a fixed distance $d$ is defined by $4$ offsetted control points $p_{i+4} = p_i-d\,n_i$ and the same weights $w_{i+4}=w_i$, for each $i=0,1,2,3$. 
They naturally define DC cubes/systems.
\end{lemma}
\begin{proof}
See e.g. \cite[Corollary 6.2]{ZubeKras2015}.
\end{proof}

\begin{lemma}\label{lem:offset-condition}
A DC system is an offset DC system if and only if at least one of its surfaces degenerates to a point.
\end{lemma}
\begin{proof}
Note that one family of coordinate lines of a system, obtained by offsetting a Dupin cyclide, is composed of straight lines that meet at infinity. The infinity is, of course, a degenerate coordinate surface of this DC system. Conversely, if one coordinate surface degenerates to one point, then all coordinate lines of the orthogonal family pass through this point. The inversion with the center at that point will map these coordinate lines to straight lines, which is the characteristic property of an offset DC system.
\end{proof}

\begin{theorem}
Any non-spherical offset DC system has exactly $2$ M-spheres in different families. 
They can be reduced by M\"obius transformations to the canonical form where the singular locus is one of the following:
\begin{itemize}
\item[$(O1)$] focal ellipse and hyperbola on orthogonal planes;
\item[$(O2)$] focal parabolas on orthogonal planes.
\end{itemize}  
\end{theorem}

\begin{proof}
It is enough to investigate the properties of the offset cube from a generic Dupin cyclide. 
A quartic Dupin cyclide, up to Euclidean similarities and offsetting, 
can be reduced to a 2-horn cyclide symmetric with respect to the planes $z=0$ and $y=0$. There is one parameter family of such Dupin cyclides and can be defined by the control points $p_0=-\ii$, $p_1=h\ii$, $p_2=-h\ii$, $p_3=\ii$, $h\neq 0$ and tangent vectors $v_1=\jj$, $v_2=\kk$ at $p_0$.
We build the offset cube based on the offset of this Dupin cyclide at a distance $d=1$. A homogeneous representation of this cube is given by
\begin{equation*}
\begin{pmatrix}
u_i \\
w_i
\end{pmatrix}_{i=0\dots7}
\!\!\!\!\!\!
=
\begin{pmatrix}
-\ii & h\jj & h\kk & h & -2\ii & (h+1)\jj & (h-1)\kk & 0 \\ 
1 & -\kk & -\jj & -h\ii & 1 & -\kk & -\jj & -h\ii
\end{pmatrix}.
\end{equation*}
The implicit equation in the $u$-direction is given by
\[(x^2+y^2+z^2+h-u^2)^2-((1+h)x-u(1-h))^2-4hy^2=0.\]
This is indeed a Dupin cyclide: by substituting $a=(1+h)/2$, $f=(1-h)/2$ and $r=u$, the quartic Dupin cyclide (see \cite{chandru1989geometry}) is obtained.
Using Remark~\ref{rem:DCSphereCond}, the spherical conditions in 3 directions are 
\begin{align*}
\sigma_1(s)=(1-h)(1-s)s,\quad
\sigma_2(t)=-(1+h)(1-t)t,\quad
\sigma_3(u)=2h.
\end{align*}
Each of the solutions $s=0$ and $s=1$ of $\sigma_1(s)=0$ gives the plane $y=0$. Similarly, each of the solutions $t=0$ and $t=1$ of $\sigma_2(t)=0$ gives the plane $z=0$. 
Of course, in the offset direction, we have $\sigma_3(u)=h\neq 0$. However, the family degenerates to the single point $P=\infty$.

The singular curves of the offset DC system are obtained by intersecting the coordinate surfaces in the non-offset directions (which are circular cones) with the coordinate planes:
\begin{align*}
 y =  x^2/\left(\frac{1-h}{2}\right)^2-z^2/h-1=0, \quad
 z =   x^2/\left(\frac{1+h}{2}\right)^2+y^2/h-1=0.
\end{align*}
Generically, they are focal ellipse/hyperbola. This is the case $(O1)$.
The limit cases $h=\pm 1$ emphasis that the basis Dupin cyclide is a torus. The offset DC system is spherical since $\sigma_1$ or $\sigma_2$ is identically zero.

We consider the same approach for parabolic cyclides for the case $(O2)$. Parabolic Dupin cyclides are equivalent under M\"obius and offset transformations. Consider the one defined by the control points $p_0=\infty$, $p_1=\ii$, $p_2=-3\ii$ and $p_3=\ii$, and $v_1=\jj$, $v_2=\kk$. 
A homogeneous representation of the offset cube (defined by the other offset layer at $d=1$) is
\begin{equation*}
\begin{pmatrix}
u_i \\
w_i
\end{pmatrix}_{i=0\dots7}
\!\!\!\!\!\!
=
\begin{pmatrix}
1 & -\kk & -3\jj & -4 \ii & 1 & -2\kk & -2\jj& 0\\ 
0 & -\jj & -\kk & -4 & 0 & -\jj & -\kk & -4
\end{pmatrix}.
\end{equation*}
The implicit equation in the $u$ direction is indeed a parabolic Dupin cyclide
\[z^2(x+3-u)+y^2(x-1-u)+(x+3-u)(x-1-u)(x-1+u)=0.\]
The spherical conditions 3 directions are computed as
\begin{align*}
\sigma_1(s)=-(1-s)s,\quad
\sigma_2(t)=(1-t)t,\quad
\sigma_3(u)=-4.
\end{align*}
The two solutions of $\sigma_1(s)=0$ (resp. $\sigma_2(t)=0$) give the same plane $y=0$ (resp. $z=0$).
The singular curves of the offset DC system are two focal parabolas defined by
\begin{align*}
 &y = z^2-8(x+1)=0,\\
 &z = y^2+8(x-1)=0.
\end{align*}
Note that in addition to the offset based construction on M-spheres, the offsets of a circular cone or cylinder form a spherical DC system.
\end{proof}

\begin{figure}[H]
    \centering
    \includegraphics[width=6cm]{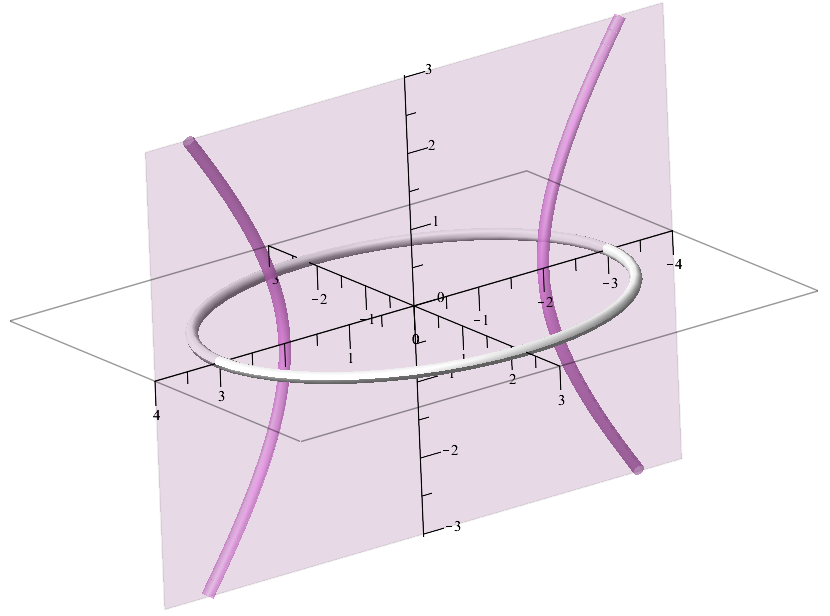} \quad
    \includegraphics[width=6cm]{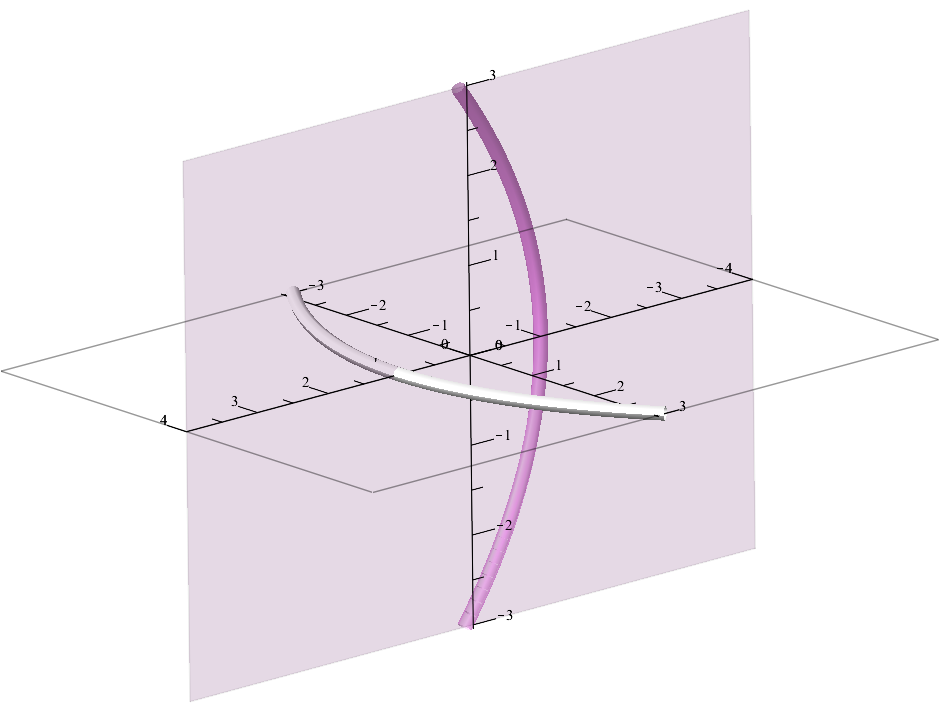}
    \caption{Focal conics, ellipse/hyperbola on the left and two parabolas on the right, as singular locus of an offset DC system.}
    \label{fig:Sings-Offset}
\end{figure}

\section{Case A: three real M-spheres of symmetry}
\label{sec:3P}

This section considers a big class of DC systems with 3 M-spheres in different families.  
These 3 M-spheres are necessarily symmetry M-spheres of the DC system and are mutually orthogonal.
Since their intersection contains exactly 2 points, two cases can appear: at least one of the intersection points is regular, and both intersecting points are singularities of the DC system. 
Each of these cases will be addressed in the following Theorems~\ref{th:3P} and \ref{th:3P-exc}. 

\begin{lemma}
\label{lem:Miquel3P}
All DC systems symmetric w.r.t. the planes $z=0$, $y=0$ and $x=0$, and have a regular point at the origin or on infinity, can be parametrized using the homogeneous control points
\begin{equation}
\label{eq:3Pcontrols}   
\begin{pmatrix}
 0 &  \ii & \jj & \ii+\jj & \kk & \ii + \kk & \jj+\kk & d+\ii+\jj+\kk\\ 
 1 &  1 & 1 & 1-a\kk & 1 &  1-c\jj & 1-b\ii & 1-b\ii-c\jj-a\kk
\end{pmatrix},
\end{equation}
where $a,b,c\in \R$ and $d=a+b+c$. The corresponding parametrization is 
\begin{equation}
\label{eq:3P-QB}
 F(s,t,u)=(dstu+s\ii+t\jj+u\kk)(1-btu\ii-csu\jj-ast\kk)^{-1}.   
\end{equation}

\end{lemma}
\begin{proof}
The construction of a DC cube in this canonical position is based on the choice of 
2-dimensional DC systems on each plane.  Hence, we repeat the general construction in Example \ref{exa:bipolar}, with different parameters $a,b,c$ on each plane; 
see the second step of DC cube construction in Figure \ref{fig:unit_cube}.
This gives all homogeneous control points except $(p_7w_7,w_7)$, where the computation of 
the Miquel point is required.
To compute $p_7$, we apply the inversion $\inv_0^1$. The 3 points $p_1$, $p_2$, $p_4$ are preserved, and they form a triangle for the Miquel's theorem. The points on the sides are 
\begin{align*}
q_3 &=
\frac{1+a}{2}p_1 + \frac{1-a}{2}p_2,\;
q_5 &=
\frac{1+c}{2}p_1 + \frac{1-c}{2}p_4,\;
q_6 &=
\frac{1+b}{2}p_2 + \frac{1-a}{2}p_4.
\end{align*}
Using Lemma \ref{lem:miquel}, the Miquel point is given by
\[
M =\frac{(1+a-c+bd)p_1+(1-a+b+cd)p_2+(1-b+c+ad)p_4}{3+d^2}.
\]
Therefore, we have
\begin{align*}
p_7 
= \mathrm{Inv}_0^1(M)
=\frac{1+a-c+bd}{1+a^2+b^2+c^2} \ii +\frac{1-a+b+cd}{1+a^2+b^2+c^2} \jj + \frac{1-b+c+ad}{1+a^2+b^2+c^2}\kk.
\end{align*}
 Next, we use the procedure in Lemma \ref{lem:MiquelFarin} to compute that $w_7=1-b\ii-c\jj-a\kk$.
The product $p_7w_7$ reduces to the simple expression
$p_7w_7=d+\ii+\jj+\kk$. The parametrization $F(s,t,u)$ is obtained by applying the formula \eqref{eq:DC-cube-QB} to the obtained homogeneous control points.
The regularity of those systems at the origin or on infinity follows from Theorem~\ref{th:3P} below.
\end{proof}

\begin{theorem}
\label{th:3P}
Any non-spherical DC system with $3$ M-spheres of symmetry and characterized by the existence of a regular point on the intersection of those M-spheres, can be reduced to a canonical form, where the singular locus is one of the following:
\begin{itemize}
\item[$(A1)$] three focal $1$-oval bicircular quartics on orthogonal planes,
\item[$(A2)$] two focal $2$-oval bicircular quartics on orthogonal planes,
\item[$(A3)$] focal ellipse and hyperbola on orthogonal planes. 
\end{itemize}
\end{theorem}
\begin{proof}
It is enough to characterize all cases of singularities of the canonical forms defined in Lemma \ref{lem:Miquel3P}.
The quadratic polynomials for spherical conditions in $s,t,u$-directions by Remark~\ref{rem:DCSphereCond} are 
\begin{align}
\label{eq:spherical-cond_abc-3P}
\sigma_1(s)=s(a+c), \quad  \sigma_2(t)=t(a+b), \quad \sigma_3(u)=u(b+c).  
\end{align}
Hence, a degeneration to a spherical DC system appears in the case $a+b=0$,  $a+c=0$, or $b+c=0$. 
The roots $s=0$, $t=0$ and $u=0$ give the degeneration with the coordinate surfaces to $x=0$, $y=0$ and $z=0$ respectively. 
The singular locus of the DC system is obtained by intersecting the degenerated planes respectively with 
the surfaces in $s,t,u$-directions; see Lemma~\ref{lem:DC-systems}. 
This gives a collection of 3 bicircular quartics:
\begin{align}
\label{eq:BQs-1}
BQ_1:&\; z = abc(x^2+y^2)^2+(ab-cd)x^2+(bd-ac)y^2-d=0,\\
\label{eq:BQs-2}
BQ_2:&\; y = abc(x^2+z^2)^2+(ad-bc)x^2+(ac-bd)z^2-d=0,\\
\label{eq:BQs-3}
BQ_3:&\; x = abc(y^2+z^2)^2+(bc-ad)y^2+(cd-ab)z^2-d=0,
\end{align}
where $d=a+b+c$. 
If $abc=0$, then we have focal ellipse and hyperbola on two planes and a conic without real points on the third plane, depending on the signs of the non-zero parameters. On the other hand, if $d=0$, then all the bicircular quartics have a singularity at the origin. By applying the unit inversion at the origin, we obtain focal ellipse and hyperbola on orthogonal planes and a conic without real points on the third plane depending on the signs of the parameters $a,b,c$. 

Assume now that $abc\neq 0$, $d\neq 0$ and let $\delta=d/abc$. We apply further a scaling by a parameter $\lambda$ on $(x,y,z)$. This is equivalent to scaling the parameters $(a,b,c)$ by $\lambda^2$. By an appropriate choice of $\lambda$, we can further assume that $\delta=\pm1$. The equations of the bicircular quartics accordingly reduce to
\begin{align}
BQ_1:&\; z = (x^2+y^2)^2+(\delta c+1/c)x^2-(\delta b+1/b)y^2+\delta=0, \label{eq:B1}\\
BQ_2:&\; y = (x^2+z^2)^2-(\delta a+1/a)x^2+(\delta b+1/b)z^2+\delta=0,
\label{eq:B2}\\
BQ_3:&\; x = (y^2+z^2)^2+(\delta a+1/a)y^2-(\delta c+1/c)z^2+\delta=0.
\label{eq:B3}
\end{align}
Those curves coincide with the focal bicircular quartics in Definition~\ref{def:foc} with the change of coefficients in \eqref{eq:foc-param}. The singular locus is then three focal 1-oval bicircular quartics when $\delta=-1$; and three focal 2-oval bicircular quartics when $\delta=+1$ and one of the curves has no real points; see Figure~\ref{fig:Sings-3P}.
\end{proof}

\begin{figure}[H]
    \centering
    \includegraphics[width=5cm]{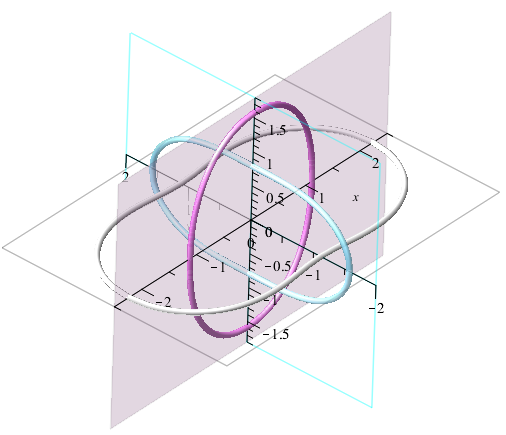} 
    \includegraphics[width=6cm]{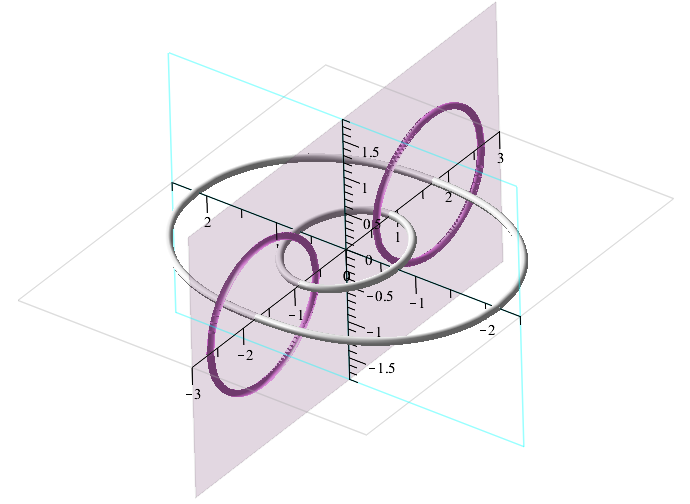}
    \caption{Singularities of a DC system with three planes of symmetry. Three focal 1-oval bicircular quartics on the left; and three focal 2-oval bicircular quartics on the right, where one of the curves has no real points.}
    \label{fig:Sings-3P}
\end{figure}

\begin{remark}
\label{rem:twice-cover-3P}
With the DC cube defined by $F(s,t,u)$ in \eqref{eq:3P-QB}, each of the symmetry planes $z=0$, $y=0$ and $x=0$ is covered twice by two $2$-polar systems by evaluating the parameters $s$, $t$, $u$ at $0$ and at infinity.
Additionally, the DC systems in Theorem~\ref{th:3P} with focal ellipse and hyperbola cases as their singularities are not offset DC systems. This is because the information from $\sigma_1(s)$, $\sigma_2(t)$, $\sigma_3(u)$ in \eqref{eq:spherical-cond_abc-3P} give only degeneration to planes, but not a degeneration to points. 
\end{remark}



Next, consider the case when both points in the intersection of three M-spheres in the given 
DC system are singular points.

\begin{theorem}
\label{th:3P-exc}
Any non-spherical DC system with $3$ M-spheres of symmetry, where the $2$ intersection points of those M-spheres are singular,
can be reduced to the canonical form $(A4)$ with two intersecting lines in the singular locus. 
\end{theorem}
\begin{proof}
Since the system is non-spherical, both points are 1-polar singularities on certain M-spheres.
Hence, the given 3 M-spheres can be reduced to 3 orthogonal planes with the cartesian coordinates
on one plane and 1-polar system on one other with the pole at the origin.
This combination of planar parameterizations already happened in one of the spherical DC systems,
namely in the offset case based on the 1-polar plane. 

The homogeneous control points are defined as follows:
\begin{equation}
\label{eq:3Pexc-controls}   
\begin{pmatrix}
 0 & \ii & 0 & \ii & \kk & \ii + \kk & \kk & -\frac{1-a}{1+a} + \ii + \kk \vspace{0.1cm} \\ 
1 & 1 & 1 & 1+\kk & 1 & 1 & 1 - \frac{2a}{1+a}\ii  & 1 - \frac{2a}{1+a}\ii + \kk
\end{pmatrix}.
\end{equation}
If $a=0$, then we have the offset case in Fig.~\ref{fig:3P_exc01} on the left (cf. control points in \eqref{eq:offset-1-polar}).
On the right side of the same figure, one can see the deformed figure with some $a > 0$.
There are 1-polar parametrizations on the vertical symmetry planes with the same pole
at the origin. The exceptional feature of this DC system is that the parametrizations on
these faces are automatically compatible since they meet at just one point.
This is the source of unexpected degrees of freedom, namely the parameter $a$.

\begin{figure}[H]
    \centering
\includegraphics[width=5.5cm]{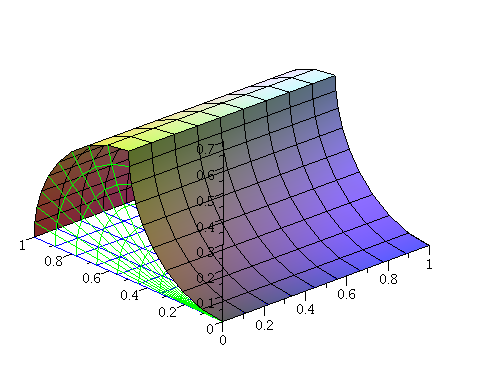}
\hskip0.5cm
\includegraphics[width=5.5cm]{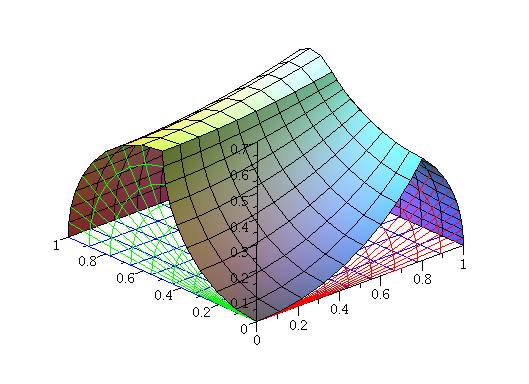}
    \caption{The offset over the 1-polar plane and its deformation.}
    \label{fig:3P_exc01}
\end{figure}
Under closer examination, one can detect two intersecting singular straight lines that
are Villarceau lines of the parabolic cyclides in coordinate families,
see Fig.~\ref{fig:3P-exception}
\end{proof}


\begin{figure}[H]
    \centering
    \includegraphics[width=6cm]{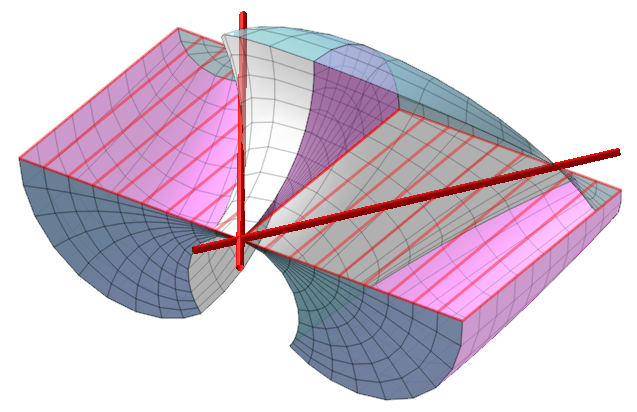}
    \caption{Two singular Villarceau lines on a DC cube.}
    \label{fig:3P-exception}
\end{figure}

\section{General non-spherical Dupin cyclidic systems}
\label{sec:General}

In this section, the most general non-spherical DC system will be constructed and
its M-spherical surfaces will be detected.
Let us note that if we have two distinct M-spheres in the same family of surfaces,
the whole family contains only M-spheres (see Section~\ref{sec:spherical}).
Therefore, a non-spherical DC system can have M-spheres only in distinct families
and they are mutually orthogonal.
The exact number of M-spheres is restricted by the following lemma.

\begin{lemma}
\label{th:General}
Any DC system contains at least two M-spheres of symmetry. 
\end{lemma}
\begin{proof}
One can always find a nonsingular point $p_0$ in $\R^3$ where three M-circles of the given 
DC systems intersect orthogonally. We assume the system is non-spherical because spherical DC systems already contain infinitely many M-spheres. Then, one can assume
these three M-circles are not mutually cospherical. After inversion at that point, they all 
go to non-intersecting straight lines with mutually orthogonal directions $\ii$, $\jj$, and $\kk$; see the left side of Figure \ref{fig:parab-cyl}. 
The corresponding DC cube can be defined with the initial 4 control points:
\[
p_0 = \infty, \ p_1 = 0, \ p_2 = \kk, \ p_4 = g \ii + h \jj, \ g, \,h \in \R. 
\]
Indeed, these lines now are not intersecting, and by scaling this
configuration of lines in $\R^3$ can be transformed to any other by 
certain \M transformation.
Then according to Theorem~\ref{th:Du-control-points}(i) 
we introduce the following three points:
\[
p_3 = p_1 (1-c) + p_2 c, \ p_6 = p_2 (1-a) + p_4 a, \ p_5 = p_4 (1-b) + p_1 b, 
\] 
and then find the Miquel point $p_7$ and all homogeneous control points
by the algorithm described in Lemma~\ref{lem:MiquelFarin}:
\begin{align*}
\begin{pmatrix}
u_i\\
w_i 
\end{pmatrix}_{i=0,\ldots,5}
&= 
\begin{pmatrix}
1 & 0 & \ii & c \kk & -h\ii+g\jj & \eta (b-1)\jj \\
0 & -\ii & -\jj & 1 & -\kk & -h+g\kk 
\end{pmatrix},\\
\begin{pmatrix}
u_6 & u_7 \\
w_6 & w_7 
\end{pmatrix}
&= 
\begin{pmatrix}
-h+(\eta a + a-1)\ii+g\kk & cg\ii+ch\jj+\eta(b-1)\kk \\
g+\jj+h\kk & (\eta+1)(a-1)+b\eta+c+h\ii-g\jj
\end{pmatrix},
\end{align*}  
where $\eta = g^2+h^2$.
Then, after computation of spherical conditions \eqref{spherical-cond}
in all three directions, we get three quadratic equations in $a,b,c$ with the 
following discriminants:
\begin{align*}
\Delta_u =& (\eta(1-b-a)-a+1)^2+4\eta b, \\
\Delta_s =& (\eta(1-b)-c)^2+4g^2c,\\
\Delta_t =& (\eta a+c+a-1)^2 + 4h^2(1-c).
\end{align*}
These discriminants define three non-intersecting parabolic cylinders in $(a,b,c)$ space, 
see Fig.~\ref{fig:parab-cyl}. 
\begin{figure}
    \centering
    \includegraphics[width=6.5cm]{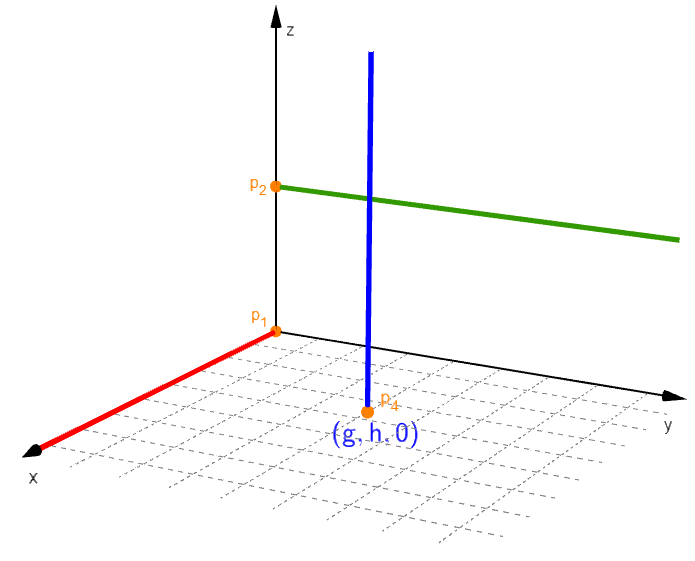}\hskip0.15cm
    \includegraphics[width=6cm]{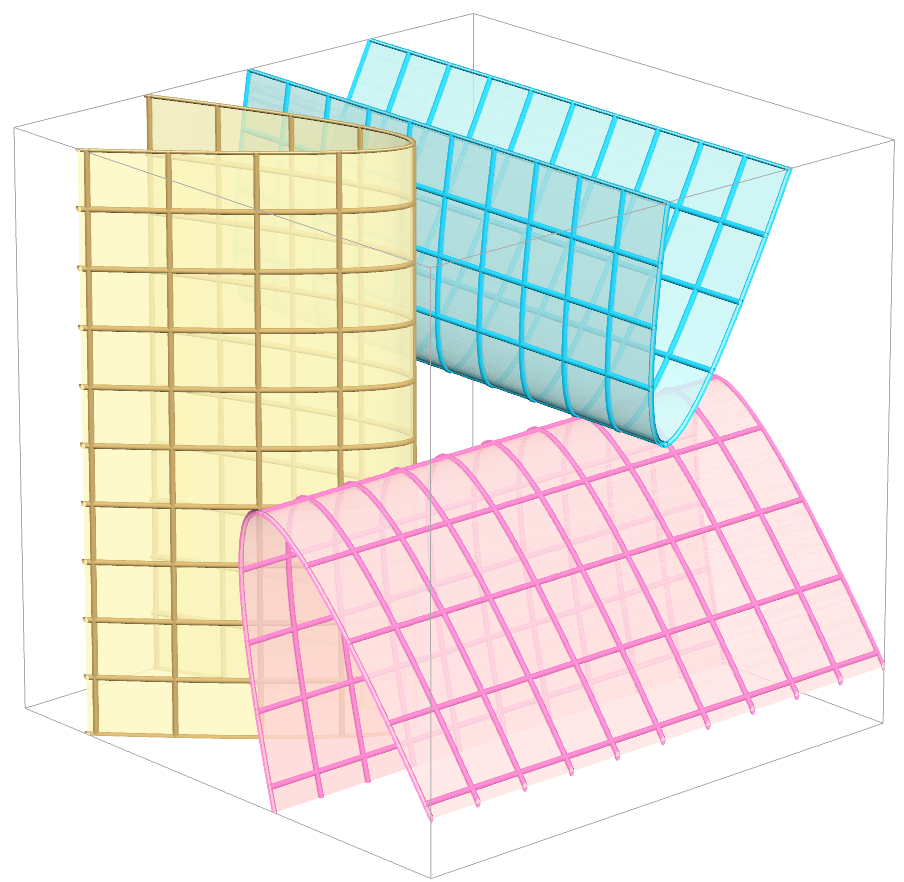}
    \caption{Left: a canonical position of 3 inverted coordinate M-circles meeting at a point. Right: 3 non-intersecting parabolic cylinders separating the $(a,b,c)$-space into 4 disjoint regions.}
    \label{fig:parab-cyl}
\end{figure}
In fact, the cylinders are separated by the 3 regions:
\begin{equation}
\label{eq:regions}
h^2/(\eta+1) \leq a \leq \eta/(\eta+1); \quad 0 \leq b \leq h^2/\eta; 
\quad 0 \leq c \leq 1.  
\end{equation}
The $(a,b,c)$-space is then separated by the 3 cylinders into $4$ regions: three insiders and one outside the cylinders. This shows that at least two of the discriminants are positive. Hence, the DC system has at least two M-spheres of symmetry. 
\end{proof}


Since a DC system has at least two real M-spheres of symmetry, starting with the two M-spheres, we clarify the separation of classes in 
Lemma~\ref{th:General} by a concrete construction with control points as in Section~\ref{sec:3P}.

\begin{theorem}
\label{th:main}
Any non-spherical DC system is \M equivalent to one of three cases:
offset (O), type (A), and type (B) having the 
symmetries:
\begin{itemize}
\item[(O)] two planes and one zero sphere of symmetry; 
\item[(A)] three planes of symmetry; 
\item[(B)] two planes and one imaginary sphere of symmetry.
\end{itemize}
The singular locus for a system of type (B) is composed of two focal $2$-oval bicircular quartics.
\end{theorem}
\begin{proof}
We apply M\"obius transformations such that the M-spheres are just the planes $z=0$ and $y=0$.
Hence, we choose the following control points
\[p_0=0,\ p_1=\ii,\ p_2=\jj,\ p_4 = \ii+\kk,\]
and the initial tangent vectors are $v_1=\ii$, $v_2 = \jj$, $v_3=\kk$ as usual. 
The point $p_4$ is a generic choice for the following reason: a coordinate M-circle, passing through the origin and orthogonal to the plane $z=0$, is either the $z$-axis or a circle on the plane $y=0$ with center on the $x$-axis. 
The first case coincides with the construction in Section~\ref{sec:3P} up to scaling, and we disregard this. In the second case, the circle has to cross the diagonal on the plane $y=0$. Hence, the choice of $p_4$ up to scaling.

In the second step of the DC cube construction, the points $p_3$, $p_5$, and $p_6$ depend on parameters $a,b,c$. We are free to choose any change of variables to get simpler symbolic results.
Indeed, we apply the change 
\[(a,b,c)\mapsto \left(\frac{1-c}{2},-2(a+c),\frac{2(1+b)}{3}\right)\] 
and obtain the homogeneous representation
\begin{align*}
\label{eq:2Pcontrols} 
\begin{pmatrix}
u_i\\
w_i
\end{pmatrix}_{i=0,\ldots,5}
&=
\begin{pmatrix}
 0 &  \ii & \jj & \ii+\jj & 2\kk & \ii + \kk  \\ 
 1 &  1 & 1 & 1-c\kk & 1+\jj &  1+(2a+2c+1)\jj
\end{pmatrix},\\
\begin{pmatrix}
u_6 & u_7 \\
w_6 & w_7
\end{pmatrix}
&=
\begin{pmatrix}
-1 + \jj + 2\kk & -2(a+b)+\ii+\jj+\kk\\ 
1+2b\ii+\jj     & 1+(2b-c)\ii+(2a+2c+1)\jj-c\kk
\end{pmatrix}.
\end{align*}
The quadratic polynomials for spherical patch detection are simply
\begin{equation*}
\sigma_1(s)=cs^2+2as+1, \quad \sigma_2(t)=-2t(b-c), \quad \sigma_3(u)=-2u(a+b+c).
\end{equation*}
The $u$-surfaces at the roots $u=0$ and $u=\infty$ of $\sigma_3$ degenerate to plane $z=0$. Similarly, the $t$-surfaces at the roots $t=0$ and $t=\infty$ of $\sigma_2$ degenerate to the plane $y=0$. 
The different cases of the theorem are obtained from the roots of $\sigma_1$ being two real, double, or two complex roots. The discriminant of $\sigma_1$ is $\Delta=a^2-c$, which defines a parabolic cylinder in the $(a,b,c)$-space.

If $\sigma_1$ has two distinct real roots, then the $s$-surfaces at the roots degenerate to one M-sphere. The DC system then has three M-spheres in different directions. This belongs to the case (A).

Assume that $\sigma_1$ has a double root, which is necessarily $s_0=-1/a$ since $c=a^2$. Then $F(s_0,t,u)=\ii/a$ for all $t,u$, which means that the $s_0$-surface degenerates to one point. This belongs to the offset case (O) by Lemma~\ref{lem:offset-condition}.

Assume now that $\sigma_1$ has complex non-real roots ($\Delta<0$). The $s$-surfaces at the roots $s=(-a\pm\sqrt{\Delta})/c$ degenerate to the imaginary double sphere
\begin{equation}
\label{eq:sphere-degen}
S: \left((x+a/c)^2+y^2+z^2 - \Delta/c^2\right)^2=0.    
\end{equation}
It is straightforward to check on the implicit equations that each surface of the DC system is preserved by inversion w.r.t. 
this imaginary sphere: 
\[\mathrm{Inv}_S(p)=-a/c\,\ii - \Delta/c^2\, (p+a/c\,\ii)^{-1}, \quad p = x\ii+y\jj+z\kk.\]
By intersecting the $t$-surfaces and $u$-surfaces with the planes $z=0$ and $y=0$, we obtain the two non-symmetric bicircular quartics as the singularities of the system:
\begin{align}
BQ_1':\; & cb(a+c)(x^2+y^2)^2-c(a-b+c)x(x^2+y^2) \nonumber\\
        & -(a^2+ab+ac+2bc+c)x^2 + (ab+ac+b^2+c^2)y^2\nonumber\\
        &-(a+b-c)x+a+b=0,\label{eq:B1p}\\
BQ_2':\; &c(4ab+4bc+c)(x^2+z^2)^2+4c(b-c)x(x^2+z^2)\nonumber\\
        &-2(2ab+2ac+4bc+c)x^2-2(2ab+2ac+2b^2+2c^2+c)z^2\nonumber\\
        &-4(b-c)x+4a+4b+1=0.\label {eq:B2p}
\end{align}
The symmetry with respect to the imaginary sphere on the DC system induces the symmetry of the curves $BQ_1'$ and $BQ_2'$ w.r.t. to the imaginary circles $S\cap \{z=0\}$ and $S\cap \{y=0\}$ respectively.  
This property implies, by Theorem \ref{th:bq-properties}, that both $BQ_1$ and $BQ_2$ are 2-oval bicircular quartics.

Let us consider next all possible degenerate cases. From the expression of the spherical conditions $\sigma_i=0$, $i=1,2,3$, our construction covers spherical DC systems when $b-c=0$ or $a+b+c=0$.
Next, if $c=0$, then $\sigma_1(s)$ is linear, and the root gives a spherical degeneration in $s$-direction. This belongs to the case $(A)$. 
Lastly, the bicircuar quartics might be singular. The singularity condition for $BQ_1'$ (resp. $BQ_2'$) is obtained by eliminating its variables $x,z$ (resp. $x,y$) from the equations defined by the partial derivatives. The found condition results to the same equation 
\[
(4a+4c+1)(4ab+4b^2+c)(b-c)^2(a+b+c)^2\Delta^2=0.
\]
This equation is unsatisfied inside the cylinder ($c>a^2$), except in the spherical DC system cases. Note further that $BQ_2'$ may degenerate to a smooth bicircular cubic if $4ab+4bc+c=0$, but no further degeneration to conics because the cubic coefficient $c(b-c)$ being zero leads to $4(a+c)c+c=0$, which cannot be satisfied inside the cylinder. However, this cubic case is just \M equivalent to two a quartic case of the curves. The focality between $BQ_1'$ and $BQ_2'$ follows from Lemma \ref{lem:type-B}. They are non-empty because they are images of real cylinders in the parameter space $(\R P^1)^3$ under the QB parametrization $F(s,t,u)$.
\end{proof}

The following lemma clarifies the situation on the canonical form of DC systems belonging to the type $(B)$. The two planes of symmetry are assumed to be $z=0$ and $y=0$ and with the prescribed symmetric 2-oval non-empty bicircular quartics ${\cal B}_1^{+}$ and ${\cal B}_2^{+}$ defined in \eqref{eq:B_1} and \eqref{eq:B_2} on those planes.

\begin{lemma}
\label{lem:type-B}
There is a unique DC system, with the singular locus ${\cal B}_1^{+}\cup {\cal B}_2^{+}$ defined by $a=k^2$ and $c=-m^2$ in \eqref{eq:foc-param}, 
that is symmetric with respect to the planes $z=0$, $y=0$, and the unit imaginary sphere $S^-: x^2+y^2+z^2+1=0$.
The corresponding DC cube is defined by the homogeneous control points
\begin{align}
\label{eq:2Pcontrols-1} 
\begin{pmatrix}
0 & \ii & 2k\jj & (k^2-1)\jj & -2m\kk & -(m^2-1)\kk & h_0 & -h_1 \\ 
1 & 0 & (k^2-1)\kk &  -2k\kk & (m^2-1)\jj & -2m\jj & h_1\ii & h_0\ii
\end{pmatrix}
\end{align}
where    
\[
h_0=2(k-m)(km+1), \quad
h_1=-\frac{((m+k)^2-(km-1)^2)(k-m)(km+1)}{(m+k)(km-1)}.
\]
\end{lemma}
\begin{proof}
We build the DC cube's control points based on 2-polar systems on the planes $z=0$ and $y=0$, with the poles symmetric with respect to the unit imaginary sphere $S^-$. 
Since the $x$-axis is already a coordinate line of the cube, we choose the first two control points as $p_0=0$ and $p_1=\infty$, and with the usual frame $\ii, \jj,\kk$ at $p_0$. Hence, the first two homogeneous control points are $(0,1)$ and $(\ii, 0)$.
Next, the two symmetric poles of one bipolar system on the plane $z=0$ are $f_1=k\ii$ and $f_2=-k^{-1}\ii$. The point $p_2$ is on the unique circle through $p_0$ and contained in the pencil of circles ${\cal C}_{12}$ defined by the two circles of zero radii at $f_1$ and $f_2$. We choose $p_2$ as the intersection of this unique circle and the $x$-axis. The point $p_3$ is on a circle through $p_0,p_1,p_2$, so it is again on the $x$-axis. Since $p_1=\infty$, the coordinate line through $p_1$ and $p_3$ must be a straight line. There is a unique line among the pencil ${\cal C}_{12}$, namely the line through the mid-point between the poles $f_1$ and $f_2$. Hence $p_3=(f_1+f_2)/2$. The points $p_2$ and $p_3$ have the expressions
\begin{equation}
p_2=-\frac{2k}{k^2-1}\ii,\quad p_3=\frac{k^2-1}{2k}\ii.
\end{equation}
The weight $w_2$ is proportional to $\kk$, but the real proportion leads to a reparametrization. Hence, we can assume that $w_2=(k^2-1)\kk$. The homogeneous control points of the face 
defined by $p_0, p_1, p_2, p_3$ 
can be identified with the formula \eqref{inf-weights} by applying the unit inversion $\inv_0^1$. From this, we have $w_3=-2k\kk$. 
The homogeneous control points $(p_iw_i,w_i)$, $i=4,5$ are constructed similarly on the plane $y=0$ with the different parameter $m$, giving
\begin{equation}
p_4=-\frac{2m}{m^2-1}\ii,\quad w_4 =(m^2-1)\jj, \quad p_5=\frac{m^2-1}{2m}\ii, \quad w_5=-2m\jj.
\end{equation}
The point $p_6$ is constructed on the M-circle through $p_0, p_2, p_4$. Hence $p_6$ is on the same $x$-axis, $p_6=\lambda\,\ii$ for some $\lambda \in \R$. The weight $w_6$ is computed using the formula \eqref{fin-weights} and adjustment of real multipliers preserving $w_2$ and $w_4$. This gives $w_6=-h_0\ii/\lambda$. 
Note that the first six control points are pairwise symmetric with respect to $S^-$; hence the last two control points $(p_iw_i,w_i)$, $i=6,7$ have the same symmetric property. Hence $(p_7w_7,w_7)=(h_0/\lambda,h_0\ii)$.

The obtained DC system is symmetric with respect to $S^-$. Indeed, we consecutively apply the following transformations to the homogeneous control points: interior reparametrization with factor $(-1,1,1)$, interchange the pairwise symmetric control points, and apply $\inv_{S^-}$ which interchanges $u_i$ and $w_i$.  The resulting homogeneous control points differ from the initial ones by the factor $\ii$. The fraction division $\pi$ cancels the latter factor, giving the same DC system.

The singular curve on the plane $z=0$ is a bicircular quartic, which is not in the symmetric form yet. However, its coefficient in $x^3$ is proportional to its coefficient in $x$, which is linear in the parameter $\lambda$. The unique value $\lambda = h_0/h_1$
gives the bicircular quartic in the symmetric form and coincides with ${\cal B}_1^{+}$. With the same $\lambda$, the singular curve of the system on the symmetry plane $y=0$ coincides with the focal curve ${\cal B}_2^{+}$.
\end{proof}

\begin{theorem}
\label{th:main2}
Non-spherical DC systems of types (O), (A), and (B) have the families of singularities with dimensions shown in brackets: 
\begin{itemize}
\item[(O)] focal conics: ellipse and hyperbola (1), or two parabolas (0);
\item[(A)] focal $1$-oval and $2$-oval bicircular quartics (2), focal ellipse and hyperbola (1), or two intersecting lines (1);  
\item[(B)] focal $2$-oval bicircular quartics (2).
\end{itemize}
Furthermore, in each of these cases, the \M class of a DC system is uniquely determined by the canonical form of its singularities.
\end{theorem}

\begin{proof}
The list of singularities in cases (O), (A), and (B) follow directly 
from Sections~\ref{sec:offsets}, 6, and Lemma~\ref{lem:type-B}, respectively.
It remains to prove the uniqueness of the constructed DC system
with a given singularity set. We will consider only two cases (A3) and (B) 
in details. Others can be examined similarly.

In the case (A3) (see Theorem~\ref{th:3P}) with focal ellipse and hyperbola, the corresponding DC cube $D_{a,b,c}$ depending on three parameters $(a,b,c)$ was presented in Lemma~\ref{lem:Miquel3P}. Its singularities 
$BQ_1 \cup BQ_3 \cup BQ_3$ have equations \eqref{eq:BQs-1}--\eqref{eq:BQs-3}, which will be conics if $abc = 0$. 
It appears that $BQ_1$ is an ellipse with focal points on $x$-axis 
which can be obtained exactly with three DC cubes $D_{a,b,c}$ 
with $p_0 = 0$, where 
$(a,b,c) = (0,T,-1), (T/(T-1),0,-1), (T/(T-1),T,0)$, $0 < T < 1$.

In the case (B) (see Lemma~\ref{lem:type-B}), the corresponding DC cube $D_{k,m}$ with the singularities containing 2-oval bicircular quartic 
${\cal B}_1^{+}$ is unique up to the choice of 2-polar systems on the symmetry planes. For example, one can choose the other poles $f'_1=k^{-1}\ii$ and $f'_2=-k\ii$ on the plane $z=0$, and similarly, there are two possibilities on the plane $y=0$. 
There will be four different DC cubes $D_{k,m}$ having $p_0=0$ with parameters $(k,m)$, $(k^{-1},m)$, $(k,m^{-1})$, $(k^{-1},m^{-1})$ 
having the same singularities.  

In both cases (A3) and (B), the numbers of the different DC cubes starting in
the same point $p_0$ correspond exactly to the number of preimages of $p_0$ as explained in Section~\ref{sec:degrees}. Hence, they are just reparametrizations of the same DC system.
\end{proof}

\section{Degrees of DC systems}
\label{sec:degrees}

This short section is devoted to degrees of DC systems that will have
important applications in the paper \cite{ErJeKr2024} about Dupin cyclide
sections.

We define the degree of a DC system $F : (\R P^1)^3 \to \widehat{\R}^3$ 
as a number of points in the preimage of a regular point 
\[
\deg F = \# F^{-1}(p), \quad p \in \widehat{\R}^3 \setminus\sing(F).
\]
This definition does not depend on the choice of the point $p$, as follows
from the lemma below.

\begin{lemma}\label{lem:degree}
For any regular points $p,q \in \widehat{\R}^3$ of a DC system $F$, the
number of preimage points at $p$ and that at $q$ are the same: $\# F^{-1}(p) = \# F^{-1}(q)$.
\end{lemma}
\begin{proof}
$F$ is a smooth map between two compact 3-dimensional differential manifolds.
By the inverse function theorem, for every regular value 
$p \in \widehat{\R}^3$ there is an open
neighborhood $U(p)$, such that its preimage $F^{-1}(U(p))$ 
is the disjoint union of a finite number of subsets $V_i$, $i=1,\ldots,n$,
where the restriction of $F$ on each $V_i$ is a diffeomorphism to $U(p)$. 
Consider the equivalence relation 
\[
p \sim q \ \Leftrightarrow \ \# F^{-1}(p) = \# F^{-1}(q)
\]
on the set of regular points $\widehat{\R}^3 \setminus \sing(F)$.
The latter set is connected since $\sing(F)$ is of codimension 2.
The corresponding equivalency classes are open disjoint sets.
Therefore, there is only one class, i.e., all preimages have the same number
of points.
\end{proof}

\begin{theorem}
DC systems $F$ of types (O), (A), and (B) depending on $\sing F$ 
have the following degrees:
\begin{table}[h!]
\begin{center}
\begin{tabular}{|l|cc|ccc|c|}
    \hline
Type & \multicolumn{2}{c|}{O} & \multicolumn{3}{c|}{A} & B\\
    \hline
$\sing F$ & EH & 2P & BQ$^{\pm}$  & EH  & 2L & BQ$^{+}$ \\   
 \hline
$\deg F$  & 4 & 3 & 4 & 3 & 2 & 4 \\
    \hline
\end{tabular}
\end{center}
\label{tab:multicol}
\end{table}

\noindent
Notations: 
{\rm BQ}$^\pm$ - focal bicircular quartics ($2$ and $1$-oval),
{\rm EH} - focal ellipse and hyperbola, 
{\rm 2P} - focal parabolas, 
{\rm 2L} - two intersecting lines.
\end{theorem}
\begin{proof}
The idea of degree counting is to choose a regular point $p$ on a symmetry plane $\Pi$ of the DC system $F$ that is parametrized by $F_i(s,t) = F(s,t,u_i)$, $i=0,1$, for two values of $u$. Each map $F_i$ defines a 2-polar (or 1-polar) system on $\Pi$ having 2 points (or 1 point) in the preimage $F_i^{-1}(p)$. Then $\deg F = \# F_1^{-1}(p) + \# F_2^{-1}(p)$. 
For example, in case of offsets of parabolic cyclide $\sing F = \mathrm{2P}$.
The symmetric planar section of the cyclide will contain a line and a circle, and their offsets will define Cartesian and classical polar systems, which are of 1-polar and 2-polar types, respectively. Hence $\deg F = 1 + 2 = 3$.     
\end{proof}
\section{Conclusions}
\label{sec:conclusions}

A natural generalization of the principal Dupin cyclide patch to a volume object,
a Dupin cyclidic (DC) cube is rationally parametrized by the fraction of 3-linear 
quaternionic polynomials.  
Actually, this construction parametrizes any triply orthogonal coordinate system
having coordinate lines circles or straight lines, which we call a DC system. 

In this paper, the full classification of such DC systems up to \M transformations 
in space, $\R^3$ is presented in the form of four big classes (here M-spheres mean
spheres or planes):
\begin{itemize}
\item[(S)] 
spherical, with a family of coordinate surfaces composed of M-spheres;
\item[(O)]
offsets, constructed as offsets of quartic and cubic Dupin cyclides;
\item[(A)]
systems with 3 M-spheres of symmetry; 
\item[(B)]
systems with 2 real M-spheres and one imaginary sphere of symmetry.
\end{itemize} 
The class (S) contains well-known classical triply orthogonal coordinate systems,
e.g., cartesian, cylindrical, conical, etc.
The class (O) was introduced in \cite{Maxwell1868} and used for separation 
of variables in the Laplace equation (see overview in \cite{SymSzeresz2022}). 
The most general classes (A) and (B) are distinguished by their singular sets, which are 
nontrivial arrangements of 1-oval and 2-oval bicircular quartic curves.
It is interesting that the same singularities appear in 19th-century books \cite{Darboux1917}
and \cite{Boecher1894} in the context of orthogonal coordinates of different kinds. 
This seems to be just the beginning of an exciting research direction.

\section*{Acknowledgements}

This work is part of a project that has received funding from the European Union’s Horizon 2020 
research and innovation programme under the Marie Skłodowska-Curie grant agreement No 860843.

\bibliographystyle{plain}
\bibliography{references_2023}
\end{document}